\documentclass[onecolumn,12pt]{IEEEtran}

\usepackage{lineno}
\usepackage{subcaption}
\usepackage{graphicx}  % remove 'demo' option for your real document
\usepackage{hyperref}
\usepackage{cuted} % for long eqns in two column papers
\usepackage{mathtools}

\usepackage{amsmath}
\usepackage{amssymb}
\usepackage{url}
\usepackage{hyperref}%
\hypersetup{colorlinks=true, linkcolor=blue, breaklinks=true, urlcolor=blue, citecolor=blue}
\usepackage{enumitem}
\usepackage{mleftright}
\usepackage{nicematrix}

\usepackage{tikz}
\usetikzlibrary{matrix,positioning,calc}
\tikzset{
    tdplot_main_coords/.style={
        x={(1cm,0cm)},
        y={(0.3cm,0.8cm)},
        z={(0cm,1cm)}
    }
}
\usetikzlibrary{decorations.pathreplacing,calc}

\usepackage{todonotes}
\usepackage{setspace}
\usepackage{comment}
\usepackage{amsthm}

\definecolor{darkgreen}{rgb}{0,0.6,0}
\def\LG{\textcolor{darkgreen}}

\newcommand{\such}{\, | \,}

\newcommand{\be}{\begin{equation}}
\newcommand{\ee}{\end{equation}}
\newcommand{\R}{\mathbb R}

\usepackage{algorithm2e,algorithmic}

\usepackage{tikz}
\usepackage{tikz-cd}
\usetikzlibrary{automata, positioning}
\usetikzlibrary{shapes,backgrounds}
\usetikzlibrary{positioning, shapes.geometric}

\usetikzlibrary{decorations.pathmorphing, arrows.meta}

\DeclareSymbolFont{bbold}{U}{bbold}{m}{n}
\DeclareSymbolFontAlphabet{\mathbbold}{bbold}

\newtheorem{Theorem}{Theorem}
\newtheorem{Lemma}[Theorem]{Lemma}
\newtheorem{Proposition}[Theorem]{Proposition}

\newtheorem{Problem}[Theorem]{Problem}

\newtheorem{Remark}[Theorem]{Remark}

\newtheorem{example}[Theorem]{Example}
\AtBeginEnvironment{example}{%
  \pushQED{\qed}%
}
\AtEndEnvironment{example}{\popQED\endexample}

\definecolor{darkgreen}{rgb}{0,0.6,0}
\def\LG{\textcolor{darkgreen}}
\def\TK{\textcolor{magenta}}

\title{
A   turnpike property in an  eigenvalue optimization~problem}
\author{Adam Kaminer, Thomas Kriecherbauer, Lars Gr\"une, and  Michael Margaliot\thanks{AK  and MM   are with the School of ECE, Tel Aviv University, 69978, Israel. TK and LG are with Mathematical Institut, University of Bayreuth, 95440 Bayreuth, Germany.
 The research of MM is partly supported by a  research grant from the Israeli Science Foundation~(ISF). The research of MM, TK, and LG was supported by DFG Grant No.\ 470999742.    Correspondence:   michaelm@tauex.tau.ac.il    }}

\begin{document}
\maketitle

\begin{abstract}
 We consider a constrained  eigenvalue optimization problem that arises in an important nonlinear dynamical model   for mRNA translation in the cell. We prove  that the ordered list of optimal parameters   admits a  turnpike property, namely, it includes three parts with the first and third part relatively short, and the values in the middle part are all approximately equal. 
 Turnpike properties have attracted considerable attention in econometrics  and optimal control theory, but to the best of our knowledge this is the first rigorous   
 proof of such a structure in an eigenvalue optimization problem.
%%%%    
\end{abstract}


\section{Introduction}
 %%%%

Eigenvalue optimization problems are important in many scientific fields 
(see, e.g.~\cite{Overton1988,BOYD199363} and the references therein) and have found  applications in 
matrix theory, mechanical vibrations,
control theory, and stability analysis. 
Here, we consider an   
eigenvalue optimization problem arising in an important nonlinear ODE model from systems biology. Our main result is a proof that the problem admits a 
turnpike property: the ordered list of optimal parameters is made of three pieces: the first and last piece are short, and the middle piece is long and the optimal parameters there are approximately equal. 

Turnpike properties received considerable attention in the context of optimal control, see \cite{FauG22} and the references therein, with roots in mathematical economy \cite{Rams28,DoSS58}. The turnpike property describes the phenomenon that optimal trajectories stay near an optimal equilibrium---the so-called turnpike equilibrium---most of the time. There are many different ways to formalize this phenomenon mathematically; we will give one of the most commonly used variants  in Eq.~\eqref{eq:turnpike} below, in order to compare it to the turnpike property we establish in the eigenvalue problem.

In optimal control, there are two main techniques to establish that the turnpike property holds: One uses a system-theoretic property called strict dissipativity, which provides an abstract energy inequality that prevents optimal solutions to stay away from the turnpike equilibrium for a long time \cite{Grue22}. The second technique uses that the turnpike equilibrium is part of an equilibrium (in an extended state space), which is hyperbolic for the dynamical system resulting from the necessary optimality conditions \cite{PorZ13,TRELAT201581}. In both cases, the turnpike equilibrium is an optimal equilibrium that is determined by solving a static optimization problem.

To the best of our knowledge,  our result in this paper is the first rigorous proof of a turnpike property in an eigenvalue optimization problem. For this problem, none of the techniques from optimal control applies directly. However, our approach in this paper bears some similarity with the second approach, as we also make use of the fact that an 
auxiliary dynamical system admits an 
hyperbolic equilibrium  point.

The remainder of this paper is organized as follows. The next section describes the eigenvalue optimization problem that we consider, and motivates it using two  applications to dynamical systems.  
Section~\ref{sec:main} 
states the  main results, and the proof is given in Section~\ref{sec:proofs}. Section~\ref{secnumerics} details several   numerical examples that demonstrate the various quantities defined  in the proofs.
The final section concludes and describes possible topics for further research.

We use standard notation. Small [capital]  letters 
are used to denote vectors [matrices]. 
$\mathbb N$ is the set of natural numbers. The non-negative orthant in~$\R^n$ is $\R^n_{\geq 0}:=\{x\in\R^n\such x_i\geq 0 \}$, and the positive orthant  is
$\R^n_{> 0}:=\{x\in\R^n\such x_i > 0 \}$. The transpose of a matrix~$A$ is~$A^\top$. For a symmetric matrix~$A\in\R^{n\times n}$, we let~$\sigma(A)$ denote the maximal eigenvalue of~$A$. We use~$1_m$ to denote the vector in~$\R^m$ with all entries one. For~$x\in\R$, $\lceil
x \rceil $ [$\lfloor x \rfloor$] is the minimal integer that is larger or equal to~$x$  [maximal integer that is smaller or equal to~$x$].

\section{Eigenvalue optimization  problem}
%%%%%%%%%%%%%%%%
Consider the~$(n+2)\times (n+2)$  symmetric and tri-diagonal  matrix
\be\label{eq:B_spec}
B(\lambda):=\begin{bmatrix}
    0 & \lambda_0^{-1/2}& 0 &0 &0 &\dots& 0 &0&0\\
 \lambda_0^{-1/2}& 0 &\lambda_1^{-1/2} &0 &0&\dots& 0 &0&0\\
  0 &\lambda_1^{-1/2} &0 &\lambda_2^{-1/2}&0&\dots& 0 &0&0\\
  &&\vdots\\
   0 &0 &0 &0&0&\dots& \lambda_{n-1}^{-1/2}& 0& \lambda_{n}^{-1/2} \\
  0 &0 &0 &0&0&\dots&0& \lambda_{n}^{-1/2} &0\\
\end{bmatrix}, 
\ee
with~$\lambda_i>0$ for~$i=0,\dots,n $. 
 
Since~$B(\lambda)$ is symmetric, all its eigenvalues are real. Since it is also componentwise  
nonnegative and irreducible,  its maximal eigenvalue (or Perron root), denoted~$\sigma(B(\lambda))$, is
simple and positive~\cite[Chapter~8]{HornJohnson2013MatrixAnalysis}. We consider the following eigenvalue optimization problem. 

\begin{Problem}\label{prob:main_prob}
%%%
Find~$\lambda \in\R^{n+1}_{> 0 }$ that minimizes~$\sigma(B(\lambda))$ subject to  the constraint
\begin{align}\label{eq:constri_lam}
  \frac1{n+1}\sum_{i=0}^{n} \lambda_i  \leq 1. 
\end{align}
\end{Problem}

This  problem admits a unique solution (see below), that  we denote    by $\bar\lambda=
\begin{bmatrix}
    \bar \lambda_0&\dots&\bar\lambda_n
\end{bmatrix}^\top$. The minimal Perron root is $\bar\sigma: = \sigma(B(\bar\lambda))$.

\begin{Remark}\label{rem:uniq}
      [Existence and uniqueness of the optimal solution]
%%%%%     
    The induced $L_2$ matrix norm of a matrix~$A$ is~$\|A\|_2=\sqrt{
    %\lambda_{\max}
    \sigma(A^\top A) }$, and if~$A$ is  symmetric  this gives~$\|A\|_2=\sigma(A )  $. Let~$e^i$, $i=1,\dots,n+2$, be the standard basis in~$\R^{n+2}$. Then
    \[
    (e^i)^\top B^\top(\lambda)B(\lambda)e^i=\begin{cases} 
    \lambda_{0}^{-1},&\text{ if } i=1,\\
    %%%%
    \lambda_{i-1}^{-1} +\lambda_{i-2}^{-1},&\text{ if } i=2,\dots,n+1,\\
    \lambda_{n}^{-1}&\text{ if } i=n+2. 
    \end{cases}
\]
Thus, if some~$\lambda_i\to 0$ then 
    $\sigma(B(\lambda))=\|B(\lambda)\|_2\to\infty$. This implies that there exists~$\eta>0$ such that the constraint in Problem~\ref{prob:main_prob} can be replaced by
    \[
    \lambda_i\geq \eta \text{ for all } i, \text { and } \frac1{n+1} \sum_{i=0}^n\lambda_i\leq  1.
    \]
This  constraint defines a compact subset in~$\R^{n+1}$, and since~$\sigma(B(\lambda))$ is a continuous function of the~$\lambda_i$s, we conclude that Problem~\ref{prob:main_prob} admits an optimal  solution~$\bar\lambda=\bar\lambda(n)$.
Furthermore, 
     since the Perron root is a monotonically non-decreasing  function of the matrix entries~\cite[Chapter~8]{HornJohnson2013MatrixAnalysis},  $\bar\lambda$ satisfies~$\frac1{n+1} 
 \sum_{i=0}^n \bar \lambda_i=1$, see also Remark \ref{remT:scaling} below. 

   An induced
   matrix norm is convex in the matrix entries.  Using this it can be shown that Problem~\ref{prob:main_prob} is strictly convex on~$\R^{n+1}_{>0}$~\cite{rfm_max}, 
    and thus  the optimal solution~$\bar\lambda$ is unique. 
%%% 
\end{Remark}

\begin{comment}
    
\textcolor{blue}{Comment just for us (can be deleted later on): let $|\cdot|$ be a vector norm and~$\|\cdot\|$ the induced matrix norm, that is, ~$\|A\|:=\max_{|x|=1} |Ax|$. Fix matrices~$A,B$ and~$r\in[0,1]$. Then
\begin{align*}
    \| r A+(1-r)B\|& = \max_{|x|=1} |(rA+(1-r)B)x |\\
    &\leq \max_{|x|=1} | rA x|+ \max_{|x|=1} |(1-r)Bx|\\
    &=r \|A\|+(1-r)\|B\|,
\end{align*} 
so~$\|\cdot\|$ is convex.
}
%%%%%%
\end{comment}

\begin{example}\label{exa:tryn100}
Fig.~\ref{fig:opt100}     depicts the optimal~$\bar\lambda_i$s for~$n=100 $ (numerically calculated  using Matlab's \emph{fmincon} procedure for constrained
optimization). It  may be seen that the optimal values
are made up of three pieces. The first [third] piece is monotonically increasing [decreasing] and short, and the middle piece is long and approximately constant. 
\end{example}

    \begin{figure}[t]
\begin{center}
%%%%%%%%%%%%%%%%%%%%
%%%%%%%%%%%%%%%%%%%%%%%[width=12cm,height=12cm]
\includegraphics[scale=0.8]{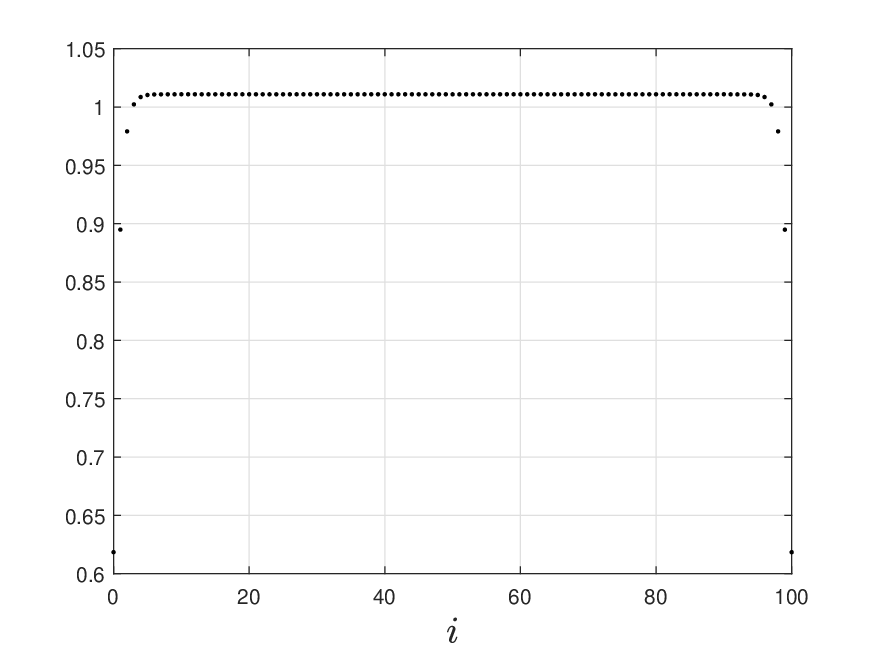}
  \caption{Optimal~$\bar\lambda_i$s  as a function of~$i$ for~$n=100$. }
  \label{fig:opt100}
%%%%%%%%%%%%%%%%%%%%%%%%%%
\end{center}
\end{figure}

\begin{Remark}\label{remT:scaling}
   We use in~\eqref{eq:constri_lam}
the upper bound $\sum_{i=0}^n \lambda_i \leq n+1$ rather than the more general    upper bound~$\sum_{i=0}^n \lambda_i \leq (n+1)c$, for some~$c>0$, because
 for any~$c>0$ we have~$B(c\lambda)=c^{-1/2}B(\lambda)$ and thus~$\sigma(B(c\lambda)) = 
 c^{-1/2}\sigma(B(\lambda)) $.
\end{Remark}

 \begin{Remark}\label{remark:sol_all_ones}
     Note that~$\lambda=1_{n+1}$ satisfies the constraints in Problem~\ref{prob:main_prob}. For this value we get that~$B(1_{n+1})$ is an~$(n+2)\times (n+2)$ tri-diagonal Toeplitz matrix and it is known~\cite{book_toeplitz} that its eigenvalues are
\[
2 \cos(\frac{k\pi}{n+3}), \quad k = 1, 2, ..., n+2, 
\]
so in particular
\be\label{eq:lam_mx_mat-ones}
\sigma(B(1_{n+1}))=  2 \cos( \frac{\pi}{n+3}).
\ee
However, we will show in Lemma \ref{lem:allonesnonopt} below that this is not the  optimal solution
for $n>1$.
%%%%
\end{Remark}

\begin{Remark}
%%%%%%
Letting~$\omega_i:=\lambda_i^{-1/2}$, $i=0,\dots,n$,  we can cast Problem~\ref{prob:main_prob} 
as follows: minimize the Perron root~$\sigma$
  of the~$(n+2) \times (n+2) $
symmetric matrix
\[
A(\omega):=\begin{bmatrix}
    0 & \omega_0& 0 &0 &0 &\dots& 0 &0&0\\
 \omega_0& 0 &\omega_1 &0 &0&\dots& 0 &0&0\\
  0 &\omega_1 &0 &\omega_2&0&\dots& 0 &0&0\\
  &&\vdots\\
   0 &0 &0 &0&0&\dots& \omega_{n-1}& 0& \omega_{n} \\
  0 &0 &0 &0&0&\dots&0& \omega_{n} &0\\
\end{bmatrix}
\]
subject to:
\begin{align}\label{eq:constri}
 \omega_i
 %\geq
 > 0 \text{ for all } i, \text{ and } \frac{1}{n+1} \sum_{i=0}^{n} \omega_i^{-2} \leq 1. 
\end{align}
%%%%
We can write~$A(\omega)$ as a linear 
%combination 
function of the~$\omega_i$s: $A(\omega)=\sum_{i=0}^n \omega_iE_i$, where~$E_i$ is the $(n+2)\times (n+2)$ matrix with all entries zero, except for entries~$(i+1,i+2)$ and~$(i+2,i+1)$, that are one. 
However, now the constraint~\eqref{eq:constri}
is nonlinear in the~$\omega_i$s. 
\end{Remark}

Our main result (Theorem~\ref{thm:mainj} below) is that for any~$n$ sufficiently
large the optimal solution~$\bar\lambda$ admits a    turnpike  property.
%%%%%
Before formally
stating it, we 
  describe two applications from dynamical systems theory 
that motivate Problem~\ref{prob:main_prob}. 

\subsection{Maximizing protein production rate in the ribosome flow model}
%%%%%%%%%%%%%%%%%%
The ribosome flow model~(RFM) is a phenomenological model for the uni-directional flow of particles 
along a 1D chain of~$n$ sites. 
It can be obtained as the dynamic mean field approximation of a fundamental model from statistical mechanics called the totally  asymmetric simple exclusion principle~\cite{solvers_guide}.

The RFM includes~$n+1$ positive parameters~$\lambda_0,\dots,\lambda_n$, where~$\lambda_i$ is  the transition rate from site~$i$ to site~$i+1$, and~$n$ state-variables~$x_1(t),\dots,x_n(t)$,
with~$x_i(t) $ describing
the density of particles in site~$i$. This is normalized such that~$x_i(t)\in[0,1]$ for all~$t$, where $x_i(t)=0$ [$x_i(t)=1$] corresponds to site~$i$ being completely empty [full].
 The state space is thus the $n$-dimensional unit cube~$[0,1]^n$. 

The RFM equations are:
\begin{align} \label{eq:rfmfull}
 \dot x_1 &= \lambda_0(1-x_1)-\lambda_1 x_1(1-x_2), \nonumber\\
 \dot x_2 &= \lambda_1x_1(1-x_2)-\lambda_2 x_2(1-x_3),\nonumber\\
 &\vdots\\
 \dot x_n &= \lambda_{n-1}x_{n-1}(1-x_n)-\lambda_n x_n.\nonumber
\end{align}
Note that defining~$x_0(t)\equiv 1$ and~$x_{n+1}(t)\equiv 0$, this can be written as
\begin{align} \label{eq:rfmshort}
 \dot x_i &= \lambda_{i-1}x_{i-1}(1-x_i)-\lambda_i x_i(1-x_{i+1}), \quad i=1,\dots,n.
\end{align}
This equation states that the change in the density in site~$i$ is the flow from site~$i-1$ to site~$i$ given by~$ \lambda_{i-1}x_{i-1}(1-x_i)$ minus the flow from site~$i$ to site~$i+1$ given by $\lambda_i x_i(1-x_{i+1})$.
Note that the last flow depends on three factors: the transition rate~$\lambda_i$, 
the density~$x_i$ in site~$i$, and the ``free space'' $(1-x_{i+1})$ in site~$i+1$.
In particular, as the density in a site increases the flow into this site decreases, and this allows to model and analyze the formation of traffic jams of particles using the~RFM. Traffic jams of ribosomes along the mRNA molecule have been implicated with various diseases (see e.g.~\cite{ribo_jam2021}).

The RFM and its variants have  been extensively used to model and analyze the flow of ribosomes along the mRNA molecule  during   translation which is a fundamental process in gene expression
 (see, e.g.,~\cite{reuveni2011genome,rfm_max,rfm_sense, alexander2017,rfmr_2015,down_reg_mrna,randon_rfm,Aditi_abortions,Aditi_extended_2022,Ortho_RFM,RFM_NEGATIVE_FEEDBACK}).
 %%%%%%%%%%%%
 Networks of interconnected RFMs have been used to model large-scale translation in the cell and the competition for shared resources~\cite{allgower_RFM,Raveh2016,nani,aditi_networks,fierce_compete}.

Ref.~\cite{margaliot2012stability} showed that the RFM admits a unique equilibrium~$e=e(\lambda_0,\dots,\lambda_n)$, with~$e\in(0,1)^n$, and that for any initial condition~$x(0)\in[0,1]^n$ the solution of the RFM converges to~$e$ (see also~\cite{Margaliot2014Entrain}).
 By~\eqref{eq:rfmfull}, the  equilibrium satisfies 
\begin{align} \label{eq:eq_psati}
%%%%%%%%%%%%%%%%%%%%%%%%%%%%%%%%%%%%
 \lambda_0    (1-e_1)   &=\lambda_1 e_1(1-e_2)\nonumber\\
 &=  \lambda_2 e_2(1-e_3)\nonumber \\
&\vdots\\
 &= \lambda_{n-2} e_{n-2}(1-e_{n-1})\nonumber\\
 &= \lambda_{n-1} e_{n-1}(1-e_n)\nonumber\\
 &= \lambda_n e_n, \nonumber
\end{align}
that is, at the equilibrium the flows into and out of each site are equal. In particular, the steady state flow out of  the last site is
\be\label{eq:def_R_steady}
R(\lambda):=\lambda_n e_n,
\ee
and this corresponds to the steady state protein production rate, which is an important biological quantity. Note that~\eqref{eq:eq_psati} and the fact that~$e\in(0,1)^n$ implies that
\be\label{eq:up_r}
R(\lambda) < \min\{\lambda_0,\dots,\lambda_n\}.
\ee

\begin{Remark}
    Eq.~\eqref{eq:eq_psati}   admits an interesting  
    graphical representation.
    To demonstrate this, consider the case~$n=3$,  
    and define~$e_0:=1$ and~$e_4:=0$.
    Then~\eqref{eq:eq_psati} becomes 
    \[
     \lambda_0   e_0 (1-e_1)    =\lambda_1 e_1(1-e_2) 
     =  \lambda_2 e_2(1-e_3) =\lambda_3 e_3(1-e_4) .
    \]
    We may view  this as an equality in the  volume of four boxes~$B_1,\dots,B_4$, 
    where~$B_i$ has     dimensions~$\lambda_{ i-1} \times e_{i-1}   \times(1- e_i)$, 
    see Fig.~\ref{fig:boxes}. 
%%%
%% 
\end{Remark}

%%%%%%%%%%%%%%%%%%%%%%%%
\begin{figure}
    \begin{center}
        
\begin{tikzpicture}[scale=1.2,tdplot_main_coords]

% === Draw blocks ===

%%FIRST BLOCK

 \draw[black] (0,0,0) -- (1,0,0) -- (1,0,1/2) -- (0,0,1/2) -- cycle;
   
 \draw[black] (0,0,0) -- (-1/2,-1/2, 1) -- (-1/2,-1/2, 3/2) -- (0,0,1/2)   ;

 \draw[black]    (-1/2,-1/2, 3/2) --   ( 1/2,-1/2, 3/2)-- (1,0,1/2) ;

\coordinate (A) at (0,0,0);
\coordinate (B) at (-1/2,-1/2,1);

% draw the line
\draw (A) -- (B);

% draw brace under it
\draw[decorate,decoration={brace,raise= 5pt}]
    (A) -- (B)
    node[midway, xshift=-18pt,yshift=-5pt] {$\scriptstyle \lambda_0$};

    \coordinate (A) at (0,0,0);
\coordinate (B) at (1,0,0);

% draw the line
\draw (A) -- (B);

% draw brace under it
\draw[decorate,decoration={brace,mirror,raise= 5pt}]
    (A) -- (B)
    node[midway, yshift=-15pt ] {$\scriptstyle e_0 $};

\coordinate (A) at (1,0,0);
\coordinate (B) at (1,0,1/2);

% draw the line
\draw (A) -- (B);

% draw brace under it
\draw[decorate,decoration={brace,mirror,raise= 5pt}]
    (A) -- (B)
    node[right , yshift=-12pt,xshift=5pt ] {$\scriptstyle 1-e_1$};
 
   \coordinate (A) at (1,0,1/2);
\coordinate (B) at (1,0,1 );

% draw the line
\draw (A) -- (B);

% draw brace under it
\draw[decorate,decoration={brace,mirror,raise= 5pt}]
    (A) -- (B)
    node[right , yshift=-12pt,xshift=5pt ] {$\scriptstyle e_1$};

   %% SECOND  BLOCK

 \draw[black] (1,0,1/2)  -- (1+3/4,0,1/2) -- (1+3/4,0,1) -- (1,0,1)  -- cycle;
   
 \draw[black] (1,0,1/2)--(0.35,-0.65,1.8) --(0.35,-0.65,2.3)  -- (1,0,1)  ;

 \draw[black]     (0.35,-0.65,2.3)  --  (0.35+3/4,-0.65,2.3)  --  (1+3/4,0,1 ) ;
 
  \coordinate (A) at (1,0,1/2);
\coordinate (B) at (0.35,-0.65,1.8);

% draw the line
\draw (A) -- (B);

% draw brace under it
\draw[decorate,decoration={brace, raise= 5pt}]
    (A) -- (B)
    node[right , yshift=-12pt,xshift=-13pt ] {$ \scriptstyle \lambda_1$};
  %% THIRD   BLOCK

 \draw[black] (1+3/4,0,1 ) -- (2,0,1) -- (2,0,3/2) -- (1+3/4,0,3/2)--cycle;
   
 \draw[black]   (1+3/4,0,1 ) -- (  -0.2, -1.95, 4.90 )-- (  -0.2, -1.95, 4.90+1/2 )--(1+3/4,0,3/2);

 \draw[black]     (  -0.2, -1.95, 4.90+1/2 )-- (  1/4-0.2, -1.95, 4.90+1/2 )-- (2,0,3/2 ) ;
  \coordinate (A) at (1+3/4,0,1);
\coordinate (B) at (2,0,1);

% draw the line
\draw (A) -- (B);

% draw brace under it
\draw[decorate,decoration={brace,mirror, raise=  2pt}]
    (A) -- (B)
    node[right , yshift=-12pt,xshift=-10pt ] {$\scriptstyle  e_2$};
 
 \coordinate (A) at (1 ,0,1);
\coordinate (B) at (1+3/4,0,1);

% draw the line
\draw (A) -- (B);

% draw brace under it
\draw[decorate,decoration={brace,   raise=  2pt}]
    (A) -- (B)
    node[midway,yshift=10pt,xshift=-9pt] {$\scriptstyle 1- e_2$};

\coordinate (A) at (1+3/4 ,0,1);
\coordinate (B) at (-0.2,-1.95,4.9);

% draw the line
\draw (A) -- (B);

% draw brace under it
\draw[decorate,decoration={brace,  mirror, raise=  2pt}]
    (A) -- (B)
    node[midway,yshift=10pt,xshift=3pt] {$\scriptstyle \lambda_2$};

     \coordinate (A) at (2 ,0,1);
\coordinate (B) at (2,0,3/2);

% draw the line
\draw (A) -- (B);

% draw brace under it
\draw[decorate,decoration={brace,  mirror,  raise=  2pt}]
    (A) -- (B)
    node[midway,xshift=15pt ] {$\scriptstyle 1- e_3$};
 %%%%%%%%

\coordinate (A) at (2 ,0,3/2);
\coordinate (B) at (2,0,2);

% draw the line
\draw (A) -- (B);

% draw brace under it
\draw[decorate,decoration={brace,  mirror, raise=  2pt}]
    (A) -- (B)
    node[midway,xshift=10pt ] {$\scriptstyle e_3$};
    
  %% FOURTH    BLOCK

 \draw[black]  (2,0,3/2)--(3,0,3/2)--(3,0,2)--(2,0,2)--cycle;

 \draw[black]  (2,0,3/2)-- (2-0.7*0.65,-0.7*0.65,3/2+0.7*1.3 ) -- (2-0.7*0.65,-0.7*0.65,3/2+0.7*1.3 +1/2) --(2,0,2);

\draw[black]  (2-0.7*0.65,-0.7*0.65,3/2+0.7*1.3 +1/2) --(3-0.7*0.65,-0.7*0.65,3/2+0.7*1.3 +1/2)--(3,0,2) ;

 \coordinate (A) at ( 2,0,3/2);
\coordinate (B) at (1.545, -0.455, 2.410);

% draw the line
\draw (A) -- (B);

% draw brace under it
\draw[decorate,decoration={brace, mirror,  raise=  2pt}]
    (A) -- (B)
    node[midway,yshift=10pt ] {$\scriptstyle \lambda_3$};
%%%%%

  \coordinate (A) at (2,0,2  );
\coordinate (B) at ( 3,0,2);

% draw the line
\draw (A) -- (B);

% draw brace under it
\draw[decorate,decoration={brace,    raise=  2pt}]
    (A) -- (B)
    node[midway,yshift=10pt,xshift=-8pt ] {$\scriptstyle 1-e_4$};
%%%%%

\end{tikzpicture}
\caption{ Graphical representation of   Eq.~\eqref{eq:eq_psati}. The volume of   the  four boxes is equal. \label{fig:boxes} }
%%%%%%%

    \end{center}
\end{figure}

Ref.~\cite{rfm_max} derived a spectral representation of the steady state values~$e$ and~$R$ in the~RFM 
(see also~\cite{EYAL_RFMD1}). Given the positive rates~$\lambda_i$, form the matrix~$B(\lambda)$ in~\eqref{eq:B_spec}, and let~$\sigma \equiv \sigma(B(\lambda))$ be its Perron root, and~$v\in\R^{n+2}_{>0}$ 
%the 
 a corresponding Perron vector. Then the steady-state production rate is
\be\label{eq:R_spect}
R(\lambda)=(\sigma(B(\lambda )))^{-2},
\ee
and the equilibrium satisfies 
\begin{align}  \label{eq:e_spet}
e_i=\frac{v_{i+2}}{\lambda_i^{1/2} \sigma  v_{i+1}},\quad i=1,\dots,n.
\end{align}
This implies that Problem~\ref{prob:main_prob} 
coincides with the problem of maximizing the steady-state production rate~$R$ subject to the ``total budget'' constraint on the transition rates in~\eqref{eq:constri_lam}.
This problem is important 
for example in  re-engineering heterologous genes where the goal is  to maximize their translation rate in a host organism (see, e.g.,~\cite{zur2020}).

\begin{example}\label{exa:n1}
    Consider the case~$n=1$. In this case the RFM reduces to the scalar equation~$\dot x_1=\lambda_0(1-x_1)-\lambda_1 x_1$, so the equilibrium is~$e_1=\frac{\lambda_0}{\lambda_0+\lambda_1}$ and the steady-state production rate is~$R=\lambda_1 e_1=  \frac{\lambda_0\lambda_1}{\lambda_0+\lambda_1}$. The matrix in~\eqref{eq:B_spec} becomes
    \[  
B =\begin{bmatrix}
    0 & \lambda_0^{-1/2}& 0  \\
 \lambda_0^{-1/2}& 0 &\lambda_1^{-1/2}  \\
   0 & \lambda_1^{-1/2}& 0
\end{bmatrix}. 
   \]
Its Perron eigenvalue  is 
$\sigma = ( \frac1{\lambda_0}+\frac1{\lambda_1} )^{1/2}$, so indeed~$R=\sigma^{-2}$. A Perron eigenvector of~$B$ is 
\[
\begin{bmatrix}
    (\lambda_1/\lambda_0)^{1/2} &   
    ( 1+\frac{\lambda_1}{\lambda_0}  )^{1/2}
     &  1
\end{bmatrix}^\top,
\]
   so~\eqref{eq:e_spet} gives
\begin{align*}
    e_1 &= \frac{1}{\lambda_1^{1/2} (\frac1{\lambda_0}+\frac1{\lambda_1})^{1/2} (1+\frac{\lambda_1}{\lambda_0})^{1/2 }}\\
    &= \frac{\lambda_0}{\lambda_0+\lambda_1}.
\end{align*}
Note that we can solve Problem~\ref{prob:main_prob} for $n=1$ due to the above formula for $\sigma$. The minimum is $\bar\sigma =\sqrt{2}$, and the unique minimizer is $\bar \lambda = 1_{2}$. This complements our main result in  Theorem~\ref{thm:mainj} that is formulated for any~$n>1$.
%%%
\end{example}

\begin{example}\label{exa:eis100}
%%%%%
Fig.~\ref{fig:optei100}     depicts the optimal~$\bar e_i$s (that is, the steady state densities corresponding to the optimal~$\bar
\lambda_i$s) for~$n=100 $. 
It  may be seen that the optimal values
are made up of three pieces. The first  and  third pieces are monotonically decreasing   and short,
and the middle piece is long with~$e_i\approx 1/2$. 
  In this case,
$\bar \sigma =1.9892$, and
the optimal steady-state production rate is~$\bar R=(\bar \sigma)^{-2}=0.2527$.  For the non-optimal vector~$\lambda=1_{101}$, we get from~\eqref{eq:lam_mx_mat-ones} that
$\sigma(1_{101})=1.9991$, so~$R(1_{101})=(\sigma(1_{101}))^{-2}= 0.2502$.
%%%%
\end{example}

    \begin{figure}[h]
\begin{center}
%%%%%%%%%%%%%%%%%%%%
%%%%%%%%%%%%%%%%%%%%%%%[width=12cm,height=12cm]
\includegraphics[scale=0.8]{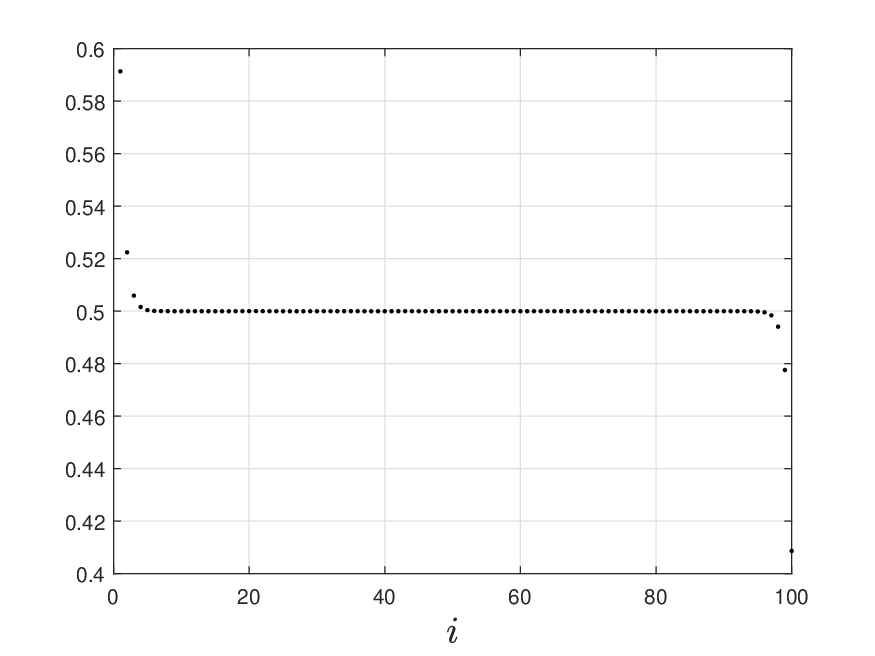}
  \caption{Optimal steady state densities~$\bar e_i$   as a function of~$i$ for~$n=100$. }
  \label{fig:optei100}
%%%%%%%%%%%%%%%%%%%%%%%%%%
\end{center}
\end{figure}

\begin{Remark}\label{rem:anot_not_opt}
Fix~$n\in\mathbb N$. Let~$c(n):= (n+1)/n$.
Define a vector
$\tilde  \lambda(n)\in\R^{n+1}$ by
\be\label{eq:def_tilde_lam}
\tilde \lambda(n):=c(n) \begin{bmatrix}
    1/2& 1&1&\dots&1&1/2
\end{bmatrix}.
\ee
Then~$\tilde\lambda$ satisfies the constraints in Problem~\ref{prob:main_prob},  and the corresponding
Perron eigenvalue and normalized
eigenvector of the matrix~$B$ are
\begin{align}\label{eq:tilde_sigma}
%%%%
\tilde \sigma:=
\sigma(B(\tilde\lambda))=2c^{-1/2}=2\left( \frac{n}{n+1} \right )^{1/2},
\end{align}
and
\[
\tilde v=(2(n+1))^{-1/2}\begin{bmatrix}
1&    \sqrt{2}&\sqrt{2}&\dots&
    \sqrt{2}&1
\end{bmatrix}^\top.
\]
Now~\eqref{eq:R_spect} gives
$\tilde R:=
R(\tilde\lambda)= \frac{n+1}{4n}$,
and
\eqref{eq:e_spet} gives
\be\label{eq:eis_nonop}
\tilde e_i=1/2,\quad i=1,\dots,n.
\ee
We   explain in Lemma \ref{lem:allonesnonopt} below why~$\tilde \lambda$ is also not the optimal solution of Problem~\ref{prob:main_prob} for any~$n>1$. 
%%%%%%%
%%
By optimality of~$\bar\lambda$, we then have
\be\label{eq:optr_lower_bound}
\bar R(n)>\tilde  R(n)=\frac{n+1}{4n}.
\ee
\end{Remark}

\begin{comment} 
******
Using~\eqref{eq:eq_psati} gives
\[
\bar\lambda_{\lfloor n/2 \rfloor} \bar e_{\lfloor n/2 \rfloor} (1-\bar e_{\lfloor n/2 \rfloor+1} )
>\frac{n+1}{4n}.
\]
For~$n$ sufficiently
large, we have~$\bar e_{\lfloor n/2 \rfloor} \approx 1/2$ and~$\bar e_{\lfloor n/2 \rfloor+1} \approx 1/2$, so we conclude that
\[
\bar\lambda_{\lfloor n/2 \rfloor}\geq 1+\frac{1}{n}.
\]
This implies that when~$n\to\infty$ the rate~$\bar \lambda_{\lfloor n/2 \rfloor}  $
 cannot converge to~$1$ at an exponential rate.
\LG{I think this conclusion depends on how fast $\bar e_{\lfloor n/2 \rfloor+1}$ converges to $1/2$. If $\bar e_{\lfloor n/2 \rfloor+1} = 1/2 + O(1/n)$, then we cannot conclude this, but if $\bar e_{\lfloor n/2 \rfloor+1} = 1/2 + O(e^{-n})$, then the statement should be true.}
******
We now derive an upper bound for~$\bar R(n)$. 
Let~$\bar e_i(n)$, $i=1,\dots,n$, denote the densities corresponding to the optimal solution~$\bar \lambda$, and define~$e_0(n):=1$, $e_{n+1}(n):=0$ for all~$n$. Then~\eqref{eq:eq_psati} yields 
\begin{align}
    \bar R= \bar \lambda_i \bar e_i (1-\bar e_{i+1}), \quad i=0,\dots,n .
\end{align}
Also, using~\eqref{eq:steady_state_symm} gives
\begin{align}\label{eq:sum_eis}
\sum_{i=0}^n \left ( \bar e_i+ (1-\bar e_{i+1} ) \right ) & = n+2.
\end{align}
%%
%%
 Consider a new optimization problem. 
Maximize~$Q=\eta_iz_i(1-z_{i+1})$, $i=0,\dots,n$,  subject to the constraints~$\eta_i>0$, $\sum_{i=0}^n\eta_i=n+1$, $z_i\in(0,1)$,
$z_0=1$, $z_{n+1}=0$, and~$z_i\in(0,1)$ for~$i=1,\dots,n$.
%%%
\end{comment}
 
 Ref.~\cite{rfm_sense} used the spectral representation of~$e$ and~$R$ to prove    that the sensitivities
of~$R$ w.r.t. a change in the rates, that is, 
 \be\label{eq:desf_sense}
 s_i(\lambda):=\frac{\partial}{\partial\lambda_i} R(\lambda),\quad i=0,1,\dots,n ,
 \ee
 %%%%%%
satisfy 
\begin{align}\label{eq:derri}
                         s_i =  \frac{2
                         R ^{ 3/2} v_{i+1}v_{i+2} }{ \lambda_i^{3/2}  },
\end{align}
where $v$ denotes the normalized Perron vector of $B(\lambda)$. Note that the $\lambda$-gradient of the constraint $\sum_{i=0}^n\lambda_i=n+1$ is given by the constant vector $1_{n+1}$ and therefore the vector 
$s(\lambda)=\begin{bmatrix}
    s_0(\lambda)&\dots&s_n(\lambda)
\end{bmatrix}^\top$
satisfies~$s(\bar \lambda)=\gamma  1_{n+1}$ 
for some Lagrange mulitplier $\gamma \in \mathbb R$. In other words, at the minimizer    all the partial derivatives $s_i(\bar \lambda)$ are equal: 
\be\label{eqT:Lagrange}
s_0(\bar \lambda) = s_1(\bar \lambda) = \ldots =s_n(\bar \lambda).
\ee

%%%%%%%%%%%%%%%%%%%%%%%%
 %%%%%

\subsection{Minimizing the maximal characteristic frequency in a linear chain}
%%%%%%%%%%%%%%%%%%%%
As  another application of Problem~\ref{prob:main_prob},
consider an ordered 1D  chain of~$N$ masses~$m_1,\dots,m_N$, with a spring 
with elasticity~$K_j$ connecting masses~$m_j$ and~$m_{j+1}$. Let~$x_j(t)$ denote the location of mass~$m_j$ at time~$t$. Then
\be\label{eq:masses_loc}
m_j \ddot x_j=K_j(x_{j+1}-x_{j})
 {+}
K_{j-1}(x_{j-1}-x_j),\quad j=1,\dots,N,
\ee
with~$K_0:=0$ and~$K_{N}:=0$. A series of coordinate transformations~\cite{dyson_chain} yields the equivalent model 
\[
\dot u = S u
\]
where~$u:[0,\infty)\to \R^{2N-1}$ is the state vector,  and~$S$
 is the tri-diagonal skew-symmetric matrix
 \[
 S:=\begin{bmatrix}
     0&p_1^{1/2}&0&\dots&0&0&0 \\
     -p_1^{1/2}&0&p_2^{1/2} &\dots &0 &0&0\\
     &&&\vdots\\
     0&0&0&\dots&-p_{2(N-1)-1}^{1/2}&0&
     p_{2(N-1)}^{1/2}\\
     %%%%%
     0&0&0&\dots&0&-p^{1/2}_{2(N-1)}&0
 \end{bmatrix} , 
 \]
with~$p_{2j-1}:=\frac{K_j}{m_j}$ and~$p_{2j}:=\frac{K_j}{m_{j+1}}$. The eigenvalues of~$S$ are the purely imaginary  values $\sqrt{-1}\omega_i$, with
\[
-\omega_{N-1} < \dots< -\omega_1<\omega_0<\omega_1<\dots<\omega_{N-1},
\]
and~$\omega_0=0$. These are the characteristic frequencies of the linear chain.
The maximal frequency $\omega_{N-1}$ is often associated to the strain and 
tear in such mechanical systems.
Ref.~\cite{min_spring} 
showed that the problem of minimizing~$\omega_{N-1}$, that is, minimizing  the maximal characteristic  frequency of the chain under the constraint
\be\label{eq:mech_constraint}
\sum_{i=1}^{N-1}\frac{m_i+m_{i+1}}{K_i}\leq N-1,
\ee
coincides with Problem~\ref{prob:main_prob}.
\begin{Remark}
    The constraint in~\eqref{eq:mech_constraint} can be explained as follows. Every couple of  two consecutive masses~$m_i,m_{i+1} $ along the chain is replaced by  a ``weighted equivalent mass''~$\frac{m_i+m_{i+1}}{K_i}$, where~$K_i$ is the elasticity of the spring that connects the two masses.
    The constraint upper bounds the   sum of all these  equivalent masses. For example, when~$N=2$, 
    Eq.~\eqref{eq:masses_loc} 
    becomes
\begin{align*}
m_1 \ddot x_1&=K_1(x_{2}-x_{1}),\\
m_2 \ddot x_2&=K_1(x_{1}-x_{2}),
\end{align*}
%%%%
and letting~$\delta x:=x_2-x_1$ gives 
\end{Remark}
\[
\frac{d^2}{dt^2} {\delta x}=-K_1\left(m_1^{-1}+m_2^{-1}\right)\delta x,
\]
so~$\omega_1=\sqrt{
K_1
\left(m_1^{-1}+m_2^{-1}\right)}$
and the optimization problem here is to minimize~$\omega_1$ subject to the constraint~$\frac{m_1+m_2}{K_1}\leq 1$.

Ref.~\cite{min_spring} proved several useful  properties of the optimal solution of Problem~\ref{prob:main_prob}. First, the optimal rates are symmetric:
\be\label{eq:symm}
\bar \lambda_i=\bar\lambda_{n-i},
\quad i=0,\dots,n,
\ee
and strictly increasing up to the middle of the chain, that is, 
\be\label{eq:lam_mono}
\bar\lambda_0<\bar\lambda_1<\dots<\bar\lambda_{\lfloor n/2 \rfloor}.
\ee
Note that these equations imply that~$\lambda_{\lfloor n/2\rfloor }$ is always the maximal rate, and if~$n$ is odd then~$\lambda_{\lfloor n/2 \rfloor}= \lambda_{\lfloor n/2 \rfloor+1}$.

Ref.~\cite{min_spring} also proved that  the corresponding  steady  state densities  in the RFM satisfy
\be\label{eq:steady_state_symm}
\bar e_i=1-\bar e_{n+1-i}, \quad i=1,\dots,n, 
\ee
and   
\be\label{eq:ratio}
\frac{\bar e_i}{1-\bar e_i} = \frac{\bar \lambda_i}{\bar \lambda_{i-1}},
\quad i=1,\dots, n.
\ee
Note that this implies in particular that if~$n$ is odd, that is,~$n=2k+1$,   then~$\bar e_{k+1}=1/2$.  

%%%%%%%%%%%
\section{Main results}\label{sec:main}
%%%%%%%%%%%%%%%%%%%%%%%
We can now state our main results. All the  proofs are given in Section~\ref{sec:proofs}.
%%%%%%%%%%

\begin{Theorem}\label{thm:mainj}
%%%%%
 Fix~$n\in\mathbb N$ with $n>1$.  Then  the optimal parameters $\bar\lambda=\bar\lambda(n)$ and the minimal Perron root~$\bar\sigma=\bar\sigma(n)$ satisfy: 
\be\label{eq:sigmbo}
  2\sqrt{1-\frac{4\ln(2)} {n+1}}<\bar \sigma<2\sqrt{1- \frac{1}{n+1} },
\ee
and
\be\label{eq:lamppp}
\frac{4}{\bar\sigma^2}\left (1-\frac{\ln(2)}{2^i}\right) < \bar\lambda_i < \frac{4}{\bar\sigma^2}, \text{ for all }  i \in\{0,1,\dots, \lfloor n/2\rfloor-1\},
\ee
%%%%%
if~$n$ is even then
%%%%
\be\label{eq:lameven}
\frac{4}{\bar\sigma^2}\left (1-\frac{4}{3}\frac{\ln(2)}{2^{n/2}}\right) < \bar\lambda_{n/2} < \frac{4}{\bar\sigma^2},  
\ee
and if~$n$ is odd then
\be\label{eq:lampoddp}
\frac{4}{\bar\sigma^2}\left (1-\frac{\ln(2)}{2^ {\lfloor n/2 \rfloor}}\right) < \bar\lambda_{\lfloor n/2 \rfloor}  =\bar\lambda_{\lfloor n/2 \rfloor+1}  < \frac{4}{\bar\sigma^2}.
\ee
%%%
Furthermore, if~$n\geq 36$ then
\be\label{eq:36}
0<\frac{4}{\bar\sigma ^2}-\bar \lambda_i<\frac1{2^i} \text{ for all } i=0,1,\dots,\lfloor n/2\rfloor.
\ee
%%%%%
%%%%%%
\end{Theorem}

We now describe    several  implications of these results. 

First, 
in using the RFM to model mRMA translation, the value of~$n$ can be several hundreds. It is thus  
important 
to know the optimal protein production rate for large values 
on~$n$. Eq.~\eqref{eq:sigmbo} implies that 
 \[
    \lim_{n\to \infty}\bar \sigma (n)=2. 
    \]
Since~$\bar R(n)=(\bar \sigma(n))^{-2}$,  this implies   that~$\lim_{n\to \infty}\bar R(n)=1/4$. 
As shown in 
Example~\ref{exa:eis100} for~$n=100 $ we already have a value close to this   limit, namely,      $\bar \sigma(100) =1.9892$, and
the optimal steady-state production rate is~$\bar R(100)=(\bar \sigma)^{-2}=0.2527$.  %For the non-optimal vector~$\lambda=1_{101}$, we get from~\eqref{eq:lam_mx_mat-ones} that
%$\sigma(1_{101})=1.9991$, so~$R(1_{101})=(\sigma(1_{101}))^{-2}= 0.2502$.
%%%%

Second, note that combining~\eqref{eq:symm} and~\eqref{eq:36} yields a kind of turnpike property. The optimal~$\bar \lambda_i$s converge exponentially fast
to the value~$4 / \bar\sigma^2 $, and in particular the optimal values in the bulk are all approximately equal to~$4 / \bar\sigma^2$. However, this is not exactly  the form of the turnpike property known in optimal control. Translated to the notation used here, this property demands the following, cf.\ \cite[Eq. (11)]{FauG22}:
%%%%%%%%%%
\begin{equation}\label{eq:turnpike}
\begin{array}{l}
\mbox{There exists a scalar }
\lambda^\star \in\R_{>0} \mbox{ such that 
the following property holds for any } 
n\in\mathbb{N}. 
\mbox{ For any } \varepsilon>0\\ 
\mbox{there is an integer }      N(\varepsilon)\in\mathbb{N} \mbox{ such that } |\bar \lambda_i-\lambda^*|<\varepsilon  \mbox{ for all } i=0,\ldots,n,
\mbox{ except for at most }   N(\varepsilon)\\
\mbox{indices.}
\end{array}
\end{equation}
%%%%%%%%%%
While Theorem \ref{thm:mainj} implies  that $\bar\lambda_i$ is close to $4/\bar\sigma^2$ most of the time in this sense, this value depends on~$n$. In contrast, \eqref{eq:turnpike} requires that $\lambda^*$ is independent of $n$. We can, however, resolve this problem by using the fact that   $4/\bar\sigma^2$ converges to $\lambda^*:=1$ as $n\to\infty$.
Hence, given $\varepsilon>0$ we can pick $N(\varepsilon)$ so large that 
%$|4/\bar\sigma^2-1|<\varepsilon/2$ 
$1 < 4/\bar\sigma^2 < 1+ \varepsilon$ 
holds for all $n\ge N(\varepsilon)$ and $\frac1{2^{N(\varepsilon)/2}}<\varepsilon$ (choosing $N(\varepsilon)$ even). This implies that the desired inequality $|\bar \lambda_i-\lambda^*|<\varepsilon$ indeed holds
for all~$i=\frac{N(\varepsilon)}{2},\ldots,n-\frac{N(\varepsilon)}{2}$, i.e., for all but at most~$N(\varepsilon)$ indices.

Third, 
Fig.~\ref{fig:opt100} shows that the maximal value of the~$\bar \lambda_i$s is larger than one for the case~$n=100$. 
Eq.~\eqref{eq:sigmbo} gives~$\frac{4}{\bar\sigma^2}>1+\frac{1}{n}$, and combining  this~\eqref{eq:lameven}
and~\eqref{eq:lampoddp}
 implies that there exists~$c>0$ such that
     \[
\bar\lambda_{\lfloor n/2\rfloor  }  > \left (1+\frac{1}{n}\right )\left(1- \frac{c}{2^{n/2}}\right).
\]
Thus, as~$n$ increases the  maximal optimal rate cannot   converge    to $\lambda^*=1$ at an exponential rate. In particular, the turnpike property \eqref{eq:turnpike} cannot be of exponential form (cf.\ \cite[eq. (18)]{FauG22}). 

\begin{comment}
    This suggests that is may be beneficial to express the turnpike property with respect to $4/\bar\sigma^2$ rather than with respect to $\lambda^*=1$, as this yields much tighter asymptotic information. We are not aware of such a phenomenon in the context of optimal control problems.}
\TK{\texttt{Couldn't we arrive at such a problem by considering the following equivalent problem. Seek $\lambda$ with minimal $\sum_{i=0}^{n} \lambda_i$ such that the Perron eigenvalue of $B(\lambda)$ equals 2?}} \LG{\texttt{Probably yes, but this would be a different problem and it is not obvious to me that our analysis still applies.}} \TK{\texttt{I think due to the rescaling of Remark~\ref{remT:scaling} these problems are mathematically equivalent. Replacing the vector $\bar \lambda$ by $\hat \lambda := (\bar \sigma/2)^2 \bar \lambda$ then the Perron eigenvalue equals $2$ and the sum of the $\hat \lambda_i$'s is minimal. By estimate~\eqref{eq:36} we now have exponential convergence of the $\hat \lambda_i$ to $\lambda^*=1$. As an optimal control problem one might want to treat this optimization problem differently. I guess that is the point of the last comment in the discussion section.}}

\end{comment}

\section{Proofs}\label{sec:proofs}
%%%%%%%%%%%%
 We begin with several auxiliary results. 
%%%%%%%
%%%%%%%%%%%%%%%

As noted in 
%Remark~\ref{remark:sol_all_ones},
Remarks~\ref{remark:sol_all_ones} and~\ref{rem:anot_not_opt}, 
for any~$n$ we have that~$\lambda=1_{n+1}$ and~$\tilde \lambda$ defined in~\eqref{eq:def_tilde_lam} both satisfy the constraints in Problem~\ref{prob:main_prob}. We require the following   result.  

\begin{Lemma}\label{lem:allonesnonopt}
    %For any~$n$, $\lambda=1_{n+1}$ is not the optimal solution of Problem~\ref{prob:main_prob}.
    For any~$n>1$, neither $\lambda=1_{n+1}$ nor~$\tilde \lambda$ defined in~\eqref{eq:def_tilde_lam} is the optimal solution of Problem~\ref{prob:main_prob}.
\end{Lemma}
\begin{proof}
%%%  
%Seeking a contradiction, assume that~$\lambda=1_{n+1}$ is the optimal solution. 
% For this value we get that
In both cases we show that condition~\eqref{eqT:Lagrange} for the minimizer is violated. For~$\lambda=1_{n+1}$ the matrix~$B(1_{n+1})$ is an~$(n+2)\times (n+2)$ tri-diagonal Toeplitz matrix and it is known~\cite{book_toeplitz} that   $\sigma(B(1_{n+1}))=  2 \cos( \frac{\pi}{n+3})$, 
and that
the corresponding normalized  eigenvector~$v\in\R^{n+2}_{>0}$ is given by 
\[
v_i= \sqrt{\frac{2}{n+3}} \sin(\frac{i \pi}{n+3}), \quad i=1,\dots, n+2. 
\]
Using the formula for the sensitivities~\eqref{eq:desf_sense} gives
\begin{align*}
    \frac{s_0(\lambda)}{s_1(\lambda)} & = \frac{v_1}{v_3} \\
    &=\frac{ \sin(\pi/(n+\textcolor{blue}{3})) }{\sin(3\pi/(n+\textcolor{blue}{3}))}\\&<1, 
\end{align*}
%%%%%  
%so~$s_0<s_1$.
for $n>1$. Hence~$s_0(\lambda)<s_1(\lambda)$ which contradicts~\eqref{eqT:Lagrange}.

Similarly we compute for $n>1$ that
\[
    \frac{s_0(\tilde \lambda)}{s_1(\tilde \lambda)}  = \frac{\tilde v_1}{\tilde v_3} 
    \frac{\tilde \lambda_1^{3/2}}{\tilde \lambda_0^{3/2}} =2 \neq 1,
 \]
 which proves the claim.
%%%
%%%
\begin{comment}
Hence, \LG{for any sufficiently  small~$\varepsilon>0$} the vector~$\hat \lambda=\begin{bmatrix}
    \hat \lambda_0&\dots& \hat \lambda_n
\end{bmatrix}^\top$   defined by
\[
\hat \lambda_i:=\begin{cases}
    1- \varepsilon, &  i=0 ,\\
    1+ \varepsilon, &  i=1 ,\\
    1, &  \text{otherwise}, 
 %%%%%   
\end{cases}
\]
satisfies the constraints in Problem~\ref{prob:main_prob} and also~$R(\hat \lambda) >R(1_{n+1})$. This contradicts the assumption that~$1_{n+1}$ is the optimal solution.    
\end{comment} 
%
\end{proof}

%\begin{Remark}\label{rem:notopt}
%Using  a similar approach shows that the vector~$\tilde \lambda$   in~\eqref{eq:def_tilde_lam} is also not the optimal solution of Problem~\ref{prob:main_prob}.    
%\end{Remark}

The next result describes   a   symmetry property   for the Perron eigenvector~$\bar v$.
\begin{Lemma}
    Let~$\bar v\in\R^{n+2}_{>0}$ be 
    the  
    normalized  Perron eigenvector of~$B(\bar \lambda)$. Then
    \be\label{eq:symm_in_v}
\bar v_i=\bar v_{n+3-i}, \quad i=1,\dots,n+2.
    \ee
\end{Lemma}
\begin{proof}
Consider the~$(n+2)\times(n+2)$ reflection matrix~$M_R:=\begin{bmatrix}
    0&0&0&\dots&0&1\\
    0&0&0&\dots&1&0\\
   &&& \vdots\\
    1&0&0&\dots&0&0 
\end{bmatrix}$. Then~\eqref{eq:symm} implies that $M_RB(\bar \lambda)=B(\bar\lambda) M_R$. Therefore,~$M_R \bar v$ is also an eigenvector of~$B(\bar\lambda)$ corresponding to the Perron root~$\bar\sigma$. Since the Perron eigenvector is unique up to scaling,~$\bar v=M_R \bar v$, and this yields~\eqref{eq:symm_in_v}. 
%%%%
\end{proof}

\subsection{Proof of Thm.~\ref{thm:mainj}}
%%%%%%%%%%%
The proof   includes several steps.
The starting point
is the spectral equation
\be\label{eq:sps}
B(\bar \lambda)\bar v =\sigma(B(\bar \lambda))\bar v, 
\ee
with
$\bar\lambda = 
\begin{bmatrix}
    \bar \lambda_0&\dots& \bar \lambda_n
\end{bmatrix}^\top>0$, 
and where~$\bar v$ is the positive Perron eigenvector,  normalized such that~$\bar v^\top\bar v=1$. 
 
 Define a vector~$\bar \omega=\bar\omega(n)$  by 
 \be\label{eq:def:wi}
 \bar
 \omega_i:=(\bar \lambda_i)^{-1/2},\quad i=0,\dots,n.
 \ee 
Then writing~\eqref{eq:sps}  explicitly yields 
\be \label{eq:av}
\bar \omega_{i-2} \bar  v_{i-1} +\bar  \omega_{i-1} \bar  v_{i+1} = \bar \sigma \; \bar  v_{i} ,\quad  i=1,\dots, n+2, 
\ee
with boundary conditions~$\bar v_0:=0$ and~$\bar v_{n+3}:=0$.
%%%%
The analysis is based on converting this set of equations
into a ``universal recursion'' that depends on a single parameter. Combining  the analysis of this recursion with\TK{~\eqref{eqT:Lagrange}},~\eqref{eq:symm} and~\eqref{eq:lam_mono} yields the results in Section~\ref{sec:main}.

We start with relation~\eqref{eqT:Lagrange}
%The proof of Lemma~\ref{lem:allonesnonopt} above implies 
that at the optimal solution all the sensitivities are equal. This yields the following result. 
\begin{Lemma}\label{lemma:sense_eq}
%%%%    
Fix~$n\in\mathbb N$. Define a vector~$\bar \mu=\bar \mu(n)$ by
\begin{align}\label{eq:def_mu}
        \bar \mu_i:=\bar v_{i+1}\bar v_{i+2}\bar \omega_i^3,\quad , i=0,\dots,n.
\end{align}
Then
\be\label{eq:all_mui_equal}
\bar \mu_0=\dots=\bar \mu_n= \bar \sigma \bar v_1^2/\bar\lambda_0.
\ee
\end{Lemma}
\begin{proof}
 %%%%   
At   the optimal solution, all the sensitivities in~\eqref{eq:derri} are equal, that is,  all the~$\bar \mu_i$s are equal. 
By~\eqref{eq:av} with~$i=1$, we have
\[
 (\bar  \lambda_{0})^{-1/2} \bar  v_{2} = \bar \sigma \bar v_1,
\]
so~$\bar v_2=   \bar \sigma\bar v_1 (\bar\lambda_0)^{1/2} $. 
Now, 
$\bar \mu_0=\bar v_1\bar v_2(\bar\lambda_0)^{-3/2} = 
\bar v_1  (\bar \sigma \bar v_1 (\bar\lambda_0)^{1/2} ) (\bar\lambda_0)^{-3/2}$,  
 and   this completes the proof. 
%%%%%%%%%
\end{proof}

\begin{comment}  
The proof  of the main result also   requires    bounds on several parameter values that are uniform in the sense that they  hold for any~$n$. 

%
Since the optimal transition rates vector~$\bar \lambda$ maximizes the steady state flow in the RFM, it is natural to expect that there exists a  lower bound~$\delta>0$ such that~$\bar\lambda_i>\delta$ for all~$i$. The next result provides an explicit lower bound.   

\begin{Lemma}\label{lem:lower_bound_lambdais}
    For any~$n$, the optimal solution~$\bar\lambda=\bar\lambda(n)$ satisfies
    \[
    \bar \lambda_i > \frac{1}{4}\left (\cos\left(\frac{\pi}{n+3}\right )\right )^{-2},\quad i=0,1,\dots,n.
    \]
\end{Lemma}
In particular, this implies that for any~$n$,  we have~$\bar \lambda_i>1/4$  for all~$i=0,\dots,n$.

\begin{proof}
%%%%%%%%%
We have
\begin{align*}
\min
_i(\bar\lambda_i)&>R(\bar\lambda)\\
&>R(1_{n+1})\\
  &  =\sigma^{-2}(B(1_{n+1})),
\end{align*}
where the first line follows from~\eqref{eq:up_r},
the second from 
the fact that~$\bar\lambda$ is the unique optimal solution, and the third from~\eqref{eq:R_spect}.
%%
Using~\eqref{eq:lam_mx_mat-ones} completes the proof. 
%%%%%
\end{proof}
\end{comment}

\subsection{Simplifying Eq.~\eqref{eq:sps}}\label{subsec:perro}
%%%%%%%%%%%%%%
%Fix~$n\in\mathbb N$.
%Let~$\bar \lambda=\bar \lambda(n)$ denote the optimal solution, with
%$\bar\lambda = 
%\begin{bmatrix}
%    \bar \lambda_0&\dots& \bar \lambda_n
%\end{bmatrix}^\top>0$. Then
%$\sum_{i=0}^n \bar{\lambda}_i=n+1$, 
%and there exist a unique positive eigenvalue~$\bar\sigma=\bar\sigma(n) $ and a unique corresponding 
%positive eigenvector $\bar v=\bar v(n)$, with~$\bar v=\begin{bmatrix}
 %\bar   v_1 & \dots& \bar v_{n+2}
%\end{bmatrix}^\top $  and~$(\bar v)^\top \bar v =1$,
%satisfying~\eqref{eq:sps}.

Define   a vector~$\bar z=\bar z(n)$  by 
\be\label{eq:def_z_i}
\bar z_i:=\bar v_i^{2/3},\quad i=1,\dots,n+2.
\ee
Then~\eqref{eq:av} becomes
\begin{align}\label{eq:obrt}
\bar \omega_{i-2}   (\bar  z_{i-1})^{3/2}  +\bar  \omega_{i-1}  (\bar  z_{i+1})^{3/2}=\bar \sigma   (\bar  z_{i})^{3/2}  ,\quad i=1,\dots, n+2,
\end{align}
with
\[
\bar z_0 =\bar z_{n+3}=0.
\]
By Lemma~\ref{lemma:sense_eq},
\begin{align*}
%%%%    
\bar \omega_i &= (\frac{\bar \mu_0}{\bar v_{i+1} \bar v_{i+2} })^{1/3} = \frac{\bar \mu_0^{1/3}}{\bar z_{i+1}^{1/2}\bar z_{i+2}^{1/2}} ,
\end{align*}
and substituting
this in~\eqref{eq:obrt} gives  
%%%%%%%%%%
 \be\label{eq:recur}
 %%%%%%%%%%%%%%
\bar z_{i-1}+ \bar z_{i+1}=   (\bar  z_{i})^{2}
\bar 
\kappa, \quad i=1,\dots,n+2,
%%%%
\ee
where we defined~$\bar \kappa=\bar \kappa
(n)$ by
\[\bar 
\kappa:=\bar \mu_0^{-1/3}\bar \sigma  =
 \bar\sigma^{2/3} \bar\lambda_0^{1/3}\bar v_1^{-2/3}   >0 .
\]
Note that  since the~$\bar v_i$s are positive and~$(\bar{ v})^\top \bar v =1$, 
\be\label{eq:z_compa}
0<\bar z_i<  1,  
\quad i=1,\dots, n+2,
\ee
where this bound holds for any value~$n$. 

\begin{comment}
    The optimal $\bar\lambda_i$s are recovered as
\begin{equation}
 \lambda_i=\frac{1}{c^2} z_i z_{i+1}. \tag{L}
\end{equation}
Moreover, summing $tz_i^2=c(z_{i-1}+z_{i+1})$ yields $t=2c^3B_n$, so
\begin{equation}
 \kappa=\frac{t}{c}=2 c^2B_n. \tag{K}
\end{equation}

\end{comment}

Define~$\bar r=\bar r(n)$ by
\be\label{eq:simpt}
\bar r:=\kappa \bar z_1
 =
 \bar\sigma^{2/3} \bar\lambda_0^{1/3}. 
\ee
Eq.~\eqref{eq:recur} with~$i=1$ yields~$\frac{\bar z_2}{\bar z_1}=\bar r$. Now~\eqref{eq:recur} with~$i=2$ 
gives~$\frac{\bar z_3}{\bar z_1} =\bar  r^3-1 $. We conclude  that
\be\label{eq:lower_bound_r}
\bar r(n)>1 \text{ for any } n.
\ee

Define a vector~$\bar a=\bar a(n)$ by 
\be\label{eq:def_ai_s}
\bar a_i:=\bar z_i/\bar z_1
=\left ( {\bar v_i}/{\bar v_1} \right  )^{2/3},\quad
i=1,\dots,n+2.
\ee
Then~$\bar a_i>0$ for any~$i=1,\dots,n+2$, and~\eqref{eq:recur} gives 
$
 \bar a_{i-1} \bar z_1+\bar a_{i+1} \bar z_1=  (\bar a_i \bar z_1)^2\bar \kappa
$, that is,
\be\label{eq:rec_ai}
 \bar a_{i-1}+\bar a_{i+1}=\bar r\bar a_i^2,\quad i=1,\dots, n+2,
\ee
with~$\bar a_0=\bar a_{n+3}=0$, and~$\bar a_1= 1$.
Also,~\eqref{eq:symm_in_v} implies that 
\be\label{eq:symm_in_ais}
\bar a_i = \bar a_{n+3-i} ,\quad i=1,\dots,n+2,  
\ee
so in particular~$\bar a_{n+2}=\bar a_1=1$. 

Note that we  can write the solution of the recursion~\eqref{eq:rec_ai}
in terms of the single unknown~$\bar r$:
\begin{align}\label{eq:iter_r}
    \bar a_0 & =0 ,\nonumber \\
    \bar a_1 &= 1,\nonumber\\
    \bar a_2 & =\bar  r,\\
    \bar a_3 &   =\bar r^3-1,\nonumber\\
    \bar a_4 &    =\bar r^7-2 \bar r^4,\nonumber\\
     \bar a_5 &     
   =\bar   r^{15}-4  \bar r^{12}+4  \bar r^9-\bar r^3+1, \nonumber
\end{align}
and so on.
Using the fact that~$\bar a_4>0$ we get that
\be\label{lower_bound_r}
\bar r(n)>2^{1/3}\approx 1.2599\text{ for any } \TK{n > 1}. 
\ee

\begin{comment}
    %%%%%%
    Moreover, for small values of~$n$ this can be used to
solve Problem~\ref{prob:main_prob}. For example, consider the case~$n=2$. Then the final condition~$\bar a_{n+3}=0$ becomes
\[
0=\bar a_5= 
r^{15} -4 r^{12}+ 4r^{9}
-r^{3} +1.
%%
\]
%%%%%
\end{comment}

%%%%%
Since~$\bar r= (\bar\sigma^{2} \bar\lambda_0)^{1/3}
$, this implies that~$\bar\sigma^2\bar\lambda_0 > 2$.

\begin{comment}
    and combining this with~\eqref{eq:upper_opt_sigma} yields
\[
\bar\lambda_0>\frac{1}{2}+\frac1{2n}\text{ for any }n.
\]
 %%%%%%
\end{comment}

 \begin{example}
%%%%%%%%
    For low orders of~$n$, the symmetry condition~\eqref{eq:symm_in_ais}
     allows to explicitly solve Problem~\ref{prob:main_prob}. For example,
    for~$n=2$, Eq.~\eqref{eq:symm_in_ais} 
    gives~$ \bar a_2=\bar a_3 $ and~\eqref{eq:iter_r} implies that~$r(2)$ is the real root of~$r^3-r-1=0$ so~$r(2)\approx 1.3247$.
%%%%
    As another example, for~$n=3$ Eq.~\eqref{eq:symm_in_ais}  gives 
    $\bar a_2=\bar a_4$ and~\eqref{eq:iter_r} 
    gives~$r(3)=(1+\sqrt{2})^{1/3}\approx 1.3415 $. 
%%%
\end{example}

The next result provides  an  upper bound for~$ \bar r$.
\begin{Lemma}
    For any~$n\in\mathbb N $ with $n>1$, we have
    \[
 \bar     r(n)<\left  (2\cos(\frac{\pi}{n+3}) \right )  ^{2/3}.
    \]
\end{Lemma}
In particular, since $\bar r(1)=2^{1/3}$ (see Example~\ref{exa:n1}),  this implies that
\be\label{eq:uppr_r}
 \bar  r(n)<2^{2/3}\approx  1.5874 \text{ for any } n.
 \ee
 \begin{proof}
     If~$\bar\lambda_0 \geq 1$ then~\eqref{eq:lam_mono} together with~\eqref{eq:symm} implies that~$\lambda_i \geq 1$ for all~$i=0,\dots,n$, and
     the constraint~\eqref{eq:constri_lam} implies that~$\bar\lambda= 1_{n+1}$ and this 
     %is a contradiction.
     contradicts the statement of Lemma~\ref{lem:allonesnonopt}.
       Thus,~$\bar \lambda_ 0 <  1$, and combining 
       this with~\eqref{eq:simpt} gives~$ \bar r<\bar\sigma^{2/3}$.  Using the definition of~$\bar \sigma$ and~\eqref{eq:lam_mx_mat-ones}
       gives $
     \bar\sigma<\sigma(B(1_{n+1}))=  2 \cos( \frac{\pi}{n+3})$, and this completes the proof.
%%%
 \end{proof}

Before further analyzing the recursion~\eqref{eq:rec_ai}, it is 
useful to relate it back to the optimal~$\bar\lambda_i$s. 
Note that  
\begin{align}\label{eq:lamb_as_ais}
%%%%%%%%%
\bar \lambda_i & =(\bar  \omega _i)^{-2} \nonumber\\
&=( \frac{\bar \mu_0^{1/3}}{\bar z_{i+1}^{1/2}\bar z_{i+2}^{1/2}} )^{-2} \nonumber\\
&=\bar \mu_0^{-2/3} \bar z_{i+1} \bar z_{i+2}\nonumber\\
%%%
&= ( \bar\sigma \bar v_1^2 / \bar \lambda_0
)^{-2/3} \bar z_1^2 
\bar a_{i+1} \bar a_{i+2}\nonumber\\
&=( \bar\sigma \bar v_1^2 / \bar \lambda_0
)^{-2/3} \bar v_1^{4/3} 
\bar a_{i+1} \bar a_{i+2}\nonumber\\
&=( \bar \lambda_0/
\bar\sigma    
)^{2/3}  
\bar a_{i+1} \bar a_{i+2} .  %%%
\end{align}
%%%%%%%
%In particular,~$\bar\lambda_0=\delta>0$, with~$\delta:=\mu^{-2/3}\bar z_1^2 r$. 
%%%%%%%%%%%%

%Combining this with the symmetry %condition~\eqref{eq:symm} gives
%\be\label{eq:symm_ais2}
%\bar a_{i+1} \bar a_{i+2} =
%\bar a_{n-i+1} \bar a_{n-i+2},\quad i=0,1,\dots,n. 
%\ee

Combining~\eqref{eq:lamb_as_ais}
  with~\eqref{eq:lam_mono} yields 
the following monotonicity result. 
%%%%%%%
\begin{Lemma}\label{lemma:mono_ai_s}
%%%%%%%%
  Fix~$n\in\mathbb N$.   If~$n=4k$ or~$n=4k+1$ then
\begin{align*}
%%%
&\bar a_1<\bar  a_3<\bar a_5<\dots< \bar a_{2k+1},  \nonumber \\
&\bar a_2<\bar a_4<\bar a_6<\dots<   \bar a_{2k+2},  
%%%
\end{align*}
%%%%%%%
    If~$n=4k+2$ or~$n=4k+3$ then
\begin{align*}
%%%
&\bar a_1<\bar  a_3<\bar a_5<\dots< \bar a_{2k+3},  \nonumber \\
&\bar a_2<\bar a_4<\bar a_6<\dots<   \bar a_{2k+2}.  
%%%
\end{align*}
%%%
\end{Lemma}

  Note that combining this  lemma  and the fact that~$\bar a_0=0$ 
  %and~$\bar a_1=1$
  and~$\bar a_2=\bar r > 1$
yields
\be\label{eq:pdi}
\bar a_{i+1}>\bar a_{i-1}, \quad i=1,\dots, \lfloor n/2\rfloor+1.
\ee

Note that~$\bar a=\bar a(n)$, as~$ \bar r= \bar r(n)$. Thus, for any~$n$ we have in general 
different sequences in Lemma~\ref{lemma:mono_ai_s}.

 \begin{comment}
%%%%%%%     Eq.~\eqref{eq:lam_mono} 
also yields a useful upper bound on the~$\bar a_i$s.
%%%%
\begin{Lemma}\label{eq:a_is_upper_bound}
%%%%%%%%%%%%%%%%%%%%%
Fix~$n\in\mathbb N$, and let~$\bar a=\bar  a(n)$ be as defined above. Then
\begin{align*}
    1    \leq \bar a_i <  4(n+1), \text{ for all } i\geq 1. 
%%%%%
\end{align*}
   \end{Lemma}
%%%%%%%%%%%
\begin{proof}
    Recall that $\bar\lambda_0>1/4$ and that, by~\eqref{eq:tilde_sigma},~$\bar\sigma <\tilde\sigma<2$.
Combining this with~\eqref{eq:lamb_as_ais} implies that
\[
\bar \lambda_i\geq (1/4) \bar a_{i+1}\bar a_{i+2}. 
\]
Since~$\sum_{i=0}^n\bar\lambda_i=n+1$, every~$\bar\lambda_i$ is upper bounded by~$n+1$. Using 
 the fact that~$\bar a_1=1$ and Lemma~\ref{lemma:mono_ai_s} completes the proof.
%%%%%
\end{proof}

\end{comment}

The next step is to
  express  the recursion~\eqref{eq:rec_ai} as a map from~$\R^2$ to~$\R^2$.

\subsection{Hyperbolic map formulation}
%Fix~$n\in\mathbb N$, and let~$r=r(n)>0$ be as defined in Section~\ref{subsec:perro}.
Define a parameter-dependent mapping $F_s:\R^2 
\to \R^2$ by 
\[
F \left ( \begin{bmatrix} u\\v\end{bmatrix} \right ) : =
\begin{bmatrix}  
 s u^2-v\\u\end{bmatrix}.
\]
Then~\eqref{eq:rec_ai} becomes the discrete-time iteration 
\be\label{eq:fmap}
   \begin{bmatrix}
       \bar a_{i+1} \\\bar a_i
   \end{bmatrix}     =
   F_{\bar r }\left ( \begin{bmatrix}
       \bar a_i\\\bar a_{i-1}
   \end{bmatrix}\right ),
\ee
so
$\begin{bmatrix}
\bar     a_{i+1}\\
\bar a_i
\end{bmatrix}
$
is an orbit of $F_{\bar r}$ emanating  from $\begin{bmatrix}
 \bar a_1\\
 \bar a_0
\end{bmatrix}=\begin{bmatrix}
  1\\
  0
\end{bmatrix}$,
and ending at $\begin{bmatrix}
 \bar a_{n+3}\\
 \bar a_{n+2}
\end{bmatrix}=\begin{bmatrix}
  0\\ 1
\end{bmatrix}$.
%%%%%%%%

%%%%%%%%%%%%%%%%%
We now analyze the asymptotic behavior of~\eqref{eq:fmap}. 
The map 
$F_s$ admits two fixed points for any~$s>0$. The first is
$\begin{bmatrix}
  0\\0
\end{bmatrix}
$, which is not relevant in our context 
because if the~$\bar a_i$s
converge to zero then the~$\bar \lambda_i$s converge to zero and this contradicts  
the argument in Remark~\ref{rem:uniq}
.  The second fixed point of~$F_s$ is
\be\label{eq:zstar}
 \begin{bmatrix}
    2/  s \\2 / s
\end{bmatrix}.
\ee
%%%%%%%
The Jacobian of~$F$ at this fixed point is~$
\begin{bmatrix}
4&-1\\1&0
\end{bmatrix}$,
%%%%
\begin{comment}
    
with eigenvalues
\[\alpha,\beta = 2ru  \mp \sqrt{2r^2u^2-1}\]
and eigenvectors 
\[
\begin{bmatrix}1\\2ku-\alpha\end{bmatrix},
\begin{bmatrix}1\\2ku-\beta\end{bmatrix}
\]
Its Jacobian at the fixed point is 
\[
DF=\begin{bmatrix}
4&-1\\1&0
\end{bmatrix},
\]
%%%%%
\end{comment}
%%%%%%%%%%%
and this has  eigenvalues 
\[
\alpha:= 2 - \sqrt{3}, \quad 
  \alpha^{-1}=2+\sqrt{3}.
\]
and    corresponding 
eigenvectors
\[
\begin{bmatrix}1\\4-\alpha\end{bmatrix},\quad 
\begin{bmatrix}1\\4-\alpha^{-1}\end{bmatrix}.
\]
Note that 
 $0<\alpha<1<\alpha^{-1}$. Thus,~\eqref{eq:zstar} 
 is a  hyperbolic fixed point of~$F$. Note that although~$F_s$ depends on~$s$, the Jacobian at the fixed point and its eigenvalues do not depend on~$s$.

%Note that~\eqref{lower_bound_r}
%implies that
%\[
%\frac{2}{r}%<2^{2/3}\approx 
% 1.5874.
%\]

%%%%%%%%%%%% 
\subsection*{Rate of convergence to the equilibrium}
%%%%%%%%%%%%%%%%%%%%%%%%%%%%%%%%%%%%%
Let
 \be\label{eq:def_qn}
 \bar  q(n):=\frac{2}{ \bar  r(n)}=\frac{2}{\bar\sigma^{2/3}
  \bar\lambda_0^{1/3} }.
 \ee
 Note that~\eqref{lower_bound_r} and~\eqref{eq:uppr_r} imply
 that
\be\label{eq:ippet_bounq}
  1.2599  \approx 2^{1/3} < \bar q(n)< 2^{2/3}\approx     1.5874\text { for any } n > 1.
\ee
%%%%%

\begin{Remark}\label{remT:Hartman}
%%%%%%
 The analysis of the map~$F_s$ suggests that the orbit $\begin{bmatrix}
 \bar a_{i+1}\\
 \bar a_{i}
\end{bmatrix}_i$ of $F_{\bar r}$ emanating  from $\begin{bmatrix}
 \bar a_1\\
 \bar a_0
\end{bmatrix}=\begin{bmatrix}
  1\\
  0
\end{bmatrix}$ approaches the hyperbolic fixed point  $\begin{bmatrix}
  \bar q\\
  \bar q
\end{bmatrix}$ at an exponential rate for~$0 \leq i \leq \lfloor (n+3)/2 \rfloor$ moving close to the stable manifold of the fixed point. For~$i > \lfloor (n+3)/2 \rfloor$ the iterates then move away from the fixed point staying close to the unstable manifold. In fact, one may use the theory of hyperbolic fixed points to prove a result similar to Theorem~\ref{thm:mainj}. As a first result, one obtains that for any~$\beta > \alpha = 2 - \sqrt{3}$ there exists a constant~$C_{\beta, \bar r} > 0$ such that 
\be\label{eqT:Hartman}
\bar q > \bar a_i > \bar q - C_{\beta, \bar r} \beta^i \quad \mbox{ for } \quad 0 \leq i \leq \lfloor (n+3)/2 \rfloor.
\ee
The core of the proof of the lower bound in~\eqref{eqT:Hartman} is the central linearization result in the Hartman-Grobman theory, see e.g.~\cite[Chapter IX]{Hartman1982}. It states that in a suitable neighborhood of a hyperbolic fixed point the dynamics is topologically equivalent to the the dynamics generated by the linear map given by the Jacobian. In our situation this means that there exists a homeomorphism $\Psi_s$ from a neighborhood $U_s$ of the fixed point onto a neighborhood $V_s$ of the origin such that for $x \in U_s$ we have 
\[ 
F_s(x) = \Psi_s^{-1} \Big( \begin{bmatrix}
4&-1\\1&0
\end{bmatrix} \Psi_s (x) \Big) .
\]
Although the dynamics generated by the linear map can easily be analyzed, the map $\Psi_s$ poses some challenges. It cannot be expected to be differentiable or Lipschitz continuous and one has to work with $\gamma$-H\"older norms where~$0 <\gamma <1$ can be chosen arbitrarily. However, the choice of~$\gamma$ has an effect on the size of the neighborhoods $U_s$ and $V_s$. This leads then to~$\beta=(2-\sqrt{3})^{\gamma}$ and a constant~$C_{\beta, \bar r}$ in~\eqref{eqT:Hartman} that is hard to evaluate but surely becomes huge if $\beta$ is chosen close to $2-\sqrt{3}$.

In view of the discussion above it is perhaps remarkable that the identities presented in this paper allow for an elementary proof of 
\[
\bar q > \bar a_i > \bar q - C_{\bar r} \beta^i \quad \mbox{ for } \quad 0 \leq i \leq \lfloor (n+3)/2 \rfloor,
\]
with~$\beta= 1/2$ and $C_{\bar r} = 2 \bar q \ln{\bar q} < 3/2$. This inequality together with relation~\eqref{eq:lamb_as_ais} are the basis for our proof of Theorem~\ref{thm:mainj}.
%%%
\end{Remark}

%The analysis of the map~$F_s$ suggests   that~$\bar a=\bar a(n)$ converges to the equilibrium~$ \bar q= \bar q(n)$ as~$n\to\infty$. 
%The next result provides  information on the convergence rate towards~$\bar q$.
%%%%%%%%%%%%%%%%%%%%%%%%%%%%%%%%%%
\begin{Lemma}\label{lemma:growth_a_is}
%%%%%%%%%%%%%%%%%%%%%%%%%%%%%%%%%%%%%%%
Fix~$n\in\mathbb N$.   Then 
   \be\label{eq:airato}
   \bar a_{i+1} >  ( \bar q)^{\frac{2^{i}-1}{2^{i}}},\quad i=1,\dots,\lfloor n/2 \rfloor   .
   \ee
   Furthermore, if $n > 1$ is odd then~\eqref{eq:airato} holds also for~$i= \lfloor n/2 \rfloor  +1$. 
  \end{Lemma}
\begin{proof}
Dividing both sides of~\eqref{eq:rec_ai}
by~$\bar a_{i-1}$ 
gives
%%%%
\be\label{eq:rec_ai+mod}
1+\frac{\bar a_{i+1}}{\bar a_{i-1}}=\frac{ \bar r\bar a_i^2}{\bar a_{i-1}}, \quad i=2,\dots,n+2.
\ee
combining this  with~\eqref{eq:pdi} 
gives 
\be\label{eq:add-bou}
\bar 
q \bar a_{i-1}<\bar a_i^2 , \quad i=2,\dots,\lfloor n/2\rfloor+1. 
\ee

We can now prove~\eqref{eq:airato}
by induction.  For~$i=1$, \eqref{eq:airato}
becomes~$  \bar r>\left(\frac{2}{ \bar r}\right)^{1/2} $,  and this indeed holds  due to~\eqref{lower_bound_r}.
Suppose that~\eqref{eq:airato}
holds with~$i$ being replaced by~$i-1$ for 
%some~$i\geq 1$. 
some~$1\leq i-1 < \lfloor n/2 \rfloor$.
Then
\begin{align}\label{eq:indcys}
    (\bar a_{i+1})^2 &\geq  \bar  q \bar a_i\nonumber\\
    &>  \bar q  (\bar  q)^{\frac{2^{i-1}-1}{2^{i-1}}}\nonumber\\
    &= (\bar  q )^{\frac{2^{i}-1}{2^{i-1}}},
\end{align}
so~$\bar a_{i+1}> (\bar q) ^{\frac{2^{i}-1}{2^{i}}}$,
and this completes the inductive proof of~\eqref{eq:airato}.

Now assume that~$n$ is odd, that is~$n=2k+1$. Let~$j:= \lfloor n/2 \rfloor  +2=k+2$. Then~\eqref{eq:symm_in_ais} gives~$\bar a_{j+1}=\bar a_{j-1}$, so~\eqref{eq:rec_ai+mod} gives
\be\label{eqT:add-bou}
\bar q \bar a_{j-1}=\bar a_j^2.
\ee
This implies that we can apply the inductive argument in~\eqref{eq:indcys} for this index as well, and this completes the proof. 
%%%
\end{proof}

If follows from~\eqref{eq:ippet_bounq}  that~$0=\bar a_0=\bar a_{n+3}<\bar q$ and that~$1=\bar a_1=\bar a_{n+2}<\bar q$.
The next result shows that in fact
all the~$\bar a_i$s are upper bounded by~$\bar q$. 
%%%%%%%%%
\begin{Lemma}\label{lemma:climbtoq}
       Fix~$n\in\mathbb N$.

       Then 
\be\label{eq:upeq}
   \bar a_{i}<\bar q, \quad i=0,\dots, n+3.
   \ee
%%%
\end{Lemma}

\begin{proof} 
Assume the claim is false. As~$\bar a_0=0$ and~$\bar a_1=1$   then there exists an index~$\ell \in\{ 2,\dots,\lfloor n/2 \rfloor+1\}$ such that 
\be\label{eq:ohid}
\bar a_\ell \geq\bar q \text{ and } \bar a_{\ell-1} <\bar q.  
\ee
The following argument only uses~$\bar a_\ell \geq\bar q$ and~$\bar a_{\ell-1} <\bar a_{\ell}$: 
\begin{align*}
\bar a_{\ell+1} &= \bar r \bar a_\ell^2-\bar a_{\ell-1}\\
&\geq \bar r \bar q \bar a_\ell-\bar a_{\ell-1} \\
&=2\bar a_\ell-\bar a_{\ell-1}\\
& > \bar a_\ell\\
&\geq \bar q.
 \end{align*}
Thus, we have~$a_{\ell+1} \geq\bar q$ and~$\bar a_{\ell} < \bar a_{\ell+1}$ which allows to repeat the argument with~$\ell$ being replaced by~$\ell+1$.
Continuing in this way, we  arrive at~$\bar a_{n+3}\geq \bar q$ and this is a contradiction, as~$\bar a_{n+3}=0$.
\end{proof}

Observe that we can combine~\eqref{eq:add-bou} and~\eqref{eqT:add-bou} in the proof of Lemma~\ref{lemma:growth_a_is} to get 
\[
q \bar a_{i-1} \leq \bar a_i^2 , \quad i=2,\dots,\lfloor (n+3)/2\rfloor. 
\]
Using in addition Lemma~\ref{lemma:climbtoq} we learn that 
\[
\bar q\bar a_{i-1} \leq \bar a_i^2<\bar q\bar a_i, \quad i=2,\dots,\lfloor (n+3)/2\rfloor,
\]
leading to an improvement of Lemma~\ref{lemma:mono_ai_s}.

\begin{Remark}
    It follows from 
    %the proof above that
    the above that
    \[
    \bar a_0<\bar a_1<\dots<\bar a_{\lfloor (n+3)/2 \rfloor} < \bar q.
    \]
    Also, the fact that~$\bar a_2<\bar q$ yields~$\bar r<2/\bar r$, so~$\bar r<\sqrt{2}$, and~$\bar q=2/\bar r >\sqrt{2}$.  
\end{Remark}
 Using the results above we can derive    bounds on the optimal~$\bar \lambda_i$s. We begin with an upper bound. 
 \begin{Lemma}\label{lemma:ipneu}
 We have
 \[
     \bar \lambda_i  <\frac{4}{\bar \sigma^2}  ,\quad i=0,1,\dots, n . 
 \]
 \end{Lemma}
    \begin{proof}
        Combining~\eqref{eq:lamb_as_ais} and~\eqref{eq:upeq} gives
     $          \bar\lambda_i < ( \bar \lambda_0/
\bar\sigma     )^{2/3} \bar q^2
        $,  and using~\eqref{eq:def_qn} completes the proof. 
    \end{proof}

The next result provides  a lower bound on the~$\bar \lambda_i$s. 
 \begin{Lemma}\label{lemm:ofgth}
     We have
     \begin{align}
         \bar \lambda_i >\frac{4}{\bar \sigma^2}(\bar q)^{\frac{-3}{2^{i+1}} },\quad i=0,1,\dots,\lfloor n/2\rfloor-1. \label{eq:piceww}
     \end{align}
     Furthermore, 
  if~$n$ is even then 
  \be\label{eq:niseveno}
\bar \lambda_{\lfloor n/2\rfloor}> \frac{4}{\bar\sigma^2}(\bar q)^\frac{-4}{2^{ ( n/2) +1}} ,
\ee
%%%%%%%%%     
and if~$n>1$ is odd then 
\[
\bar \lambda_{\lfloor n/2\rfloor }> \frac{4}{\bar\sigma^2}(\bar q)^\frac{-3}{2^{ \lfloor n/2\rfloor+1}}. 
\]
%%%%
 \end{Lemma}

\begin{proof}
%%%%%
We consider three cases. 

\noindent  Case 1. Suppose that~$i=0$. Eq.~\eqref{eq:def_qn}
gives~$\bar\lambda_0=\frac{8}{\bar\sigma^2 (\bar q)^3}$. Write this as 
\[
\bar \lambda_0=\frac{4}{\bar \sigma^2} (\bar q)^{-3/2} \cdot\frac{2}{(\bar q)^{3/2}}.
\]
By~\eqref{eq:ippet_bounq}, 
 $\frac{2}{(\bar q)^{3/2}}>1$ for $n>1$, so~$\bar \lambda_0>\frac{4}{\bar \sigma^2} (\bar q)^{-3/2}$. 

\noindent  Case 2. Suppose that~$i\in\{1,\dots,\lfloor n/2\rfloor-1\}$. Then the bound  follows from combining~\eqref{eq:lamb_as_ais},   
Lemma~\ref{lemma:growth_a_is}, and~\eqref{eq:def_qn}.

\noindent  Case 3. Suppose that~$i=\lfloor n/2\rfloor$.
If~$n$ is even then~\eqref{eq:symm_in_ais} gives~$\bar a_{i+1}=\bar a_{i+2}$, so~\eqref{eq:lamb_as_ais} 
gives
\begin{align*}
%%%%%%%%%
\bar \lambda_i
&=( \bar \lambda_0/
\bar\sigma    
)^{2/3}  
\bar a_{i+1} ^2  ,   
\end{align*}
and using the bound in Lemma~\ref{lemma:growth_a_is}
proves~\eqref{eq:niseveno}. If $n>1$ is odd then the proof follows from the last
statement in Lemma~\ref{lemma:growth_a_is}.
%%
 %%% 
\end{proof}

 We can now prove the main result. 
%%%%%%%%%%%%
Fix~$i\in\{0,1,\dots,\lfloor n/2\rfloor-1\}$. Then~\eqref{eq:piceww} gives 
%%%
\begin{align*}
         \bar \lambda_i >\frac{4}{\bar \sigma^2}e^{ \frac{-3\ln(\bar q)}{2^{i+1}} } 
%%%%%
%%%%%%
\end{align*}
Using the bound~$e^{-x}>1-x$ for any~$x>0$ yields
\be\label{eq:lambq}
\bar \lambda_i >\frac{4}{\bar\sigma^2} \left ( 1-  \frac{3\ln(\bar q)}{2^{i+1}} \right ) .
\ee
By~\eqref{eq:ippet_bounq}, we have~$\ln(\bar q)<\frac{2}{3}\ln(2)$, and this proves the inequality on the left-hand side of~\eqref{eq:lamppp}. 
     The second inequality follows  from Lemma~\ref{lemma:ipneu}.
%%%%%%%%%
Eqs.~\eqref{eq:lameven}
and~\eqref{eq:lampoddp} follow similarly from Lemma~\ref{lemm:ofgth}.

We now prove~\eqref{eq:sigmbo}. Assume first that~$n>1$ is odd. Then
\begin{align*}
    n+1&=\sum_{i=0}^n \bar\lambda_i\\
    &=2\sum_{i=0}^{(n-1)/2} \bar\lambda_i\\
    &>2 \frac{4}{\bar\sigma^2} \sum_{i=0}^{(n-1)/2}
    \left( 1-\frac{3\ln(\bar q)  }{2^{i+1}}  \right )\\
    &>\frac{4}{\bar\sigma^2} (n+1-6\ln(\bar q)), 
\end{align*}
where we used~\eqref{eq:lambq}. Now assume that~$n$ is even. Then
%%%%%
\begin{align*}
    n+1&=\sum_{i=0}^n \bar\lambda_i\\
    &=2\sum_{i=0}^{(n/2)-1} \bar\lambda_i+\bar\lambda_{n/2} \\
    &> \frac{4}{\bar\sigma^2}
    \left( 2
    \sum_{i=0}^{(n/2)-1}
    \left( 1-\frac{3\ln(\bar q)  }{2^{i+1}}  \right ) + (1-\frac{4\ln(\bar q)}{2^{(n/2)+1}} ) \right) \\
    &= \frac{4}{\bar\sigma^2}
    \left(   n+{1}-6\ln(\bar q) (1-\frac{1}{2^{n/2}})-2\ln(\bar q)\frac1{2^{n/2}}  \right) 
    \\
 &>   \frac{4}{\bar\sigma^2} (n+1-6\ln(\bar q)).
\end{align*}
%%%%%%%
Thus, in both cases we have
\be\label{eq:boundfrac}
\frac{\bar\sigma^2}{4}>1-\frac{6\ln(\bar q)}{n+1},
\ee
and using~\eqref{eq:ippet_bounq}  proves the left-hand side inequality in~\eqref{eq:sigmbo}. The inequality on the right-hand side of~\eqref{eq:sigmbo} follows from~\eqref{eq:tilde_sigma} and  Lemma~\ref{lem:allonesnonopt}.

To prove~\eqref{eq:36}, note that from the bounds we already proved, we have~$\frac{4}{\bar\sigma^2}-\bar\lambda_i <\frac{4}{\bar\sigma^2} \frac{4}{3} \frac{\ln(2)}{2^i}$. 
Using~\eqref{eq:boundfrac} gives
\begin{align*}
    \frac{4}{\bar\sigma^2} & < \frac{n+1}{n+1-6\ln(\bar q)}\\
    &< \frac{n+1}{n+1-4\ln(2)},
\end{align*}
so
\[
\frac{4}{\bar\sigma^2}-\bar\lambda_i <\frac{4}{3}\frac{n+1}{n+1-4\ln(2)} \frac{\ln(2)}{2^i}.
\]
%%%%%%
Let~$g(n):=(4/3)\frac{n+1}{n+1-4\ln(2)} {\ln(2)}$.
As~$n\to\infty$,  we have~$g(n)\to (4/3)\ln(2)<1 $. In particular, $g(n)<1$ for any~$n\geq 36$. This proves~\eqref{eq:36}. 
%%%%%%
This completes the proof of Thm.~\ref{thm:mainj}.

\section{Numerical examples}\label{secnumerics}
%%%%%%%
The following 
examples demonstrate   the  various 
quantities defined in the proofs above. 

\begin{example}\label{exa:hyper}
    Consider the case~$n=3$. Then
    \[
    B(\lambda)=\begin{bmatrix}
        0&\lambda_0^{-1/2}& 0 & 0 &0 \\
         \lambda_0^{-1/2}& 0 & \lambda_1^{-1/2} &0 &0\\
        0&\lambda_1^{-1/2}& 0 & \lambda_2^{-1/2} &0 \\
        0&0&\lambda_2^{-1/2}& 0 & \lambda_3^{-1/2}  \\
        0&0& 0 & \lambda_3^{-1/2} &0  
    \end{bmatrix}, 
    \]
    and the  problem is  minimizing~$\sigma(B(\lambda))$ subject to~$\sum_{i=0}^3 \lambda_i\leq 4$.
A numerical solution using Matlab gives
\[
\bar \lambda=\begin{bmatrix} 0.8284   & 1.1716   & 1.1716  &  0.8284
\end{bmatrix}^\top 
\]
(all numerical values are to four digit accuracy).
The corresponding Perron eigenvalue is~$\bar \sigma=\sigma(B(\bar \lambda))=1.7071$, and the normalized Perron eigenvector 
is
\[
\bar v=\begin{bmatrix}
    0.3218&
    0.5000&
    0.5412&
    0.5000&
    0.3218
\end{bmatrix}^\top.
\]
Thus, $\bar r=\bar\sigma ^{2/3}\bar\lambda_0^{1/3}=1.3415$.
    The values in~\eqref{eq:def_mu} are
    \begin{align*}
        \bar v_1 \bar v_2 \bar \omega_0^3= 0.2134,\\
         \bar v_2 \bar v_3 \bar \omega_1^3= 0.2134,\\
         \bar v_3 \bar v_4 \bar \omega_2^3= 0.2134,\\
         \bar v_4 \bar v_5 \bar \omega_3^3= 0.2134. 
    \end{align*}
Eq.~\eqref{eq:def_ai_s} gives 
\[
\bar a = 
\begin{bmatrix}
 1.0000
&    1.3415
&    1.4142
&    1.3415
&    1.0000
\end{bmatrix}^\top.
\]
The relevant fixed point of~$F_{\bar r}$ is
\[
 \begin{bmatrix}
   \bar q &\bar q  
\end{bmatrix}^\top= \begin{bmatrix}
    1.4909
&    1.4909
\end{bmatrix}^\top
.
\]

\begin{comment}
The Jacobian for
$\begin{bmatrix}
\bar     z_{i+1}\\
\bar z_i
\end{bmatrix}$
is DF $ =\begin{bmatrix}
2k \bar  z_{i+1}&-1\\1&0
\end{bmatrix}, $
so the eigen values of the Jacobian matrices are
\[\alpha_i,\beta_i = 2k \bar  z_{i+1} \mp \sqrt{2k^2 \bar  z_{i+1}^2-1}\]
\[\alpha =\begin{bmatrix}
0.4473&0.3034&0.2849&0.3034&0.4473\end{bmatrix},\]
\[\beta =\begin{bmatrix}
2.2357&3.2958&3.5094&3.2959&2.2357\end{bmatrix},\]

note that $
\alpha\to 2 - \sqrt{3}, \quad 
\beta\to 2+\sqrt{3}.
$ in the middle, but do not reach them. They reach only at $ n \to \infty$.
Also, note that $\alpha$ is convex with the minimum in the middle and that $\alpha\beta=1$. 

\[\frac{d\alpha_i}{d \bar  z_{i+1}} ,\frac{ d\beta_i}{d \bar  z_{i+1}} = 2k  \mp 2k^2\bar  z_{i+1} (2k^2 \bar  z_{i+1}^2-1)^{-0.5}\]

\end{comment}

\end{example}

\begin{example}
    %%%%%
    As another example, consider the case $n=20$. 
Then the sequence of~$\bar \lambda_i$s is
\[
\begin{matrix}
0.6453 &   0.9338  &  1.0216 &    1.0459 &   1.0525  &  
1.0542 &    1.0547 &   1.0548  &  1.0549 \\   1.0549 &   1.0549  &
1.0549 &    1.0549 &   1.0548  &  1.0547 & 
 1.0542 &   1.0525 &  
 1.0459 \\  1.0216  &  
 0.9338 &   0.6453
\end{matrix}
\]
The corresponding Perron eigenvalue is~$\bar \sigma=\sigma(B(\bar \lambda))=    1.9473
$, and the normalized Perron eigenvector~$\bar v$ is
%%%%%%%%%%%%
\[
\begin{matrix}
    0.1240&
    0.1939&
    0.2158&
    0.2219&
    0.2235&
    0.2240&
    0.2241&
    0.2241&
    0.2241\\
    0.2241&
    0.2241&
    0.2241&
    0.2241&
    0.2241&
    0.2241&
    0.2241&
    0.2240&
    0.2235\\
    0.2219&
    0.2158&
    0.1939&
    0.1240
\end{matrix} 
\]
Thus, $ \bar  r=\bar\sigma ^{2/3}\bar\lambda_0^{1/3}=1.3476$. 
    The values in~\eqref{eq:def_mu} are all equal to
    \begin{align*}
        \bar v_1 \bar v_2 \bar \omega_0^3=     0.0464.
    \end{align*}
The vector~$\bar a$ in~\eqref{eq:def_ai_s} is 
\[
 \begin{matrix}
    1.0000&
    1.3476&
    1.4471&
    1.4742&
    1.4815&
    1.4834&
    1.4840&
    1.4841&
    1.4841&
    1.4842\\
    1.4842&
    1.4842&
    1.4842&
    1.4841&
    1.4841&
    1.4840&
    1.4834&
    1.4815&
    1.4742&
    1.4471\\
    1.3476&
    1.0000,
\end{matrix}
\]
and the relevant fixed point of~$F_{\bar r}$ is
\[
 \begin{bmatrix}
    \bar  q & \bar q  
\end{bmatrix}^\top= \begin{bmatrix}
     1.4842
&     1.4842
\end{bmatrix}^\top
.
\]
%%
%%
%%%%%%
\end{example}

 \begin{example}
 We know from
 Theorem~\ref{thm:mainj} 
that for any~$n\geq 36$ we have
\be\label{eq:rotp}
0<2^i \left( \frac{4}{\bar\sigma ^2(n)}-\bar \lambda_i(n)\right) <1  \text{ for all } i=0,1,\dots,\lfloor n/2\rfloor.
\ee
We calculated numerically the values
\begin{align}\label{eq:maxM}
   % m(n)&:= \min_{i=0,1,\dots,\lfloor n/2\rfloor} 2^i \left( \frac{4}{\bar\sigma ^2(n)}-\bar \lambda_i(n)\right),\\
    M(n)&:=\max_{i=0,1,\dots,\lfloor n/2\rfloor} 2^i \left( \frac{4}{\bar\sigma ^2(n)}-\bar \lambda_i(n)\right),
%%%    
\end{align}
for $n=2,\dots,35$. These values are plotted in Fig.~\ref{fig:mandm} as a function of~$n$. It may be seen that~$M(n)<1$ for all~$n\in\{2,\dots,35\}$ providing a numerical validation  that the estimate~\eqref{eq:36} in fact holds for any~$n>1$. 
%%%%%
%%%%%%
 \end{example}

\begin{figure}[t]
\centering
\includegraphics[scale=0.6]{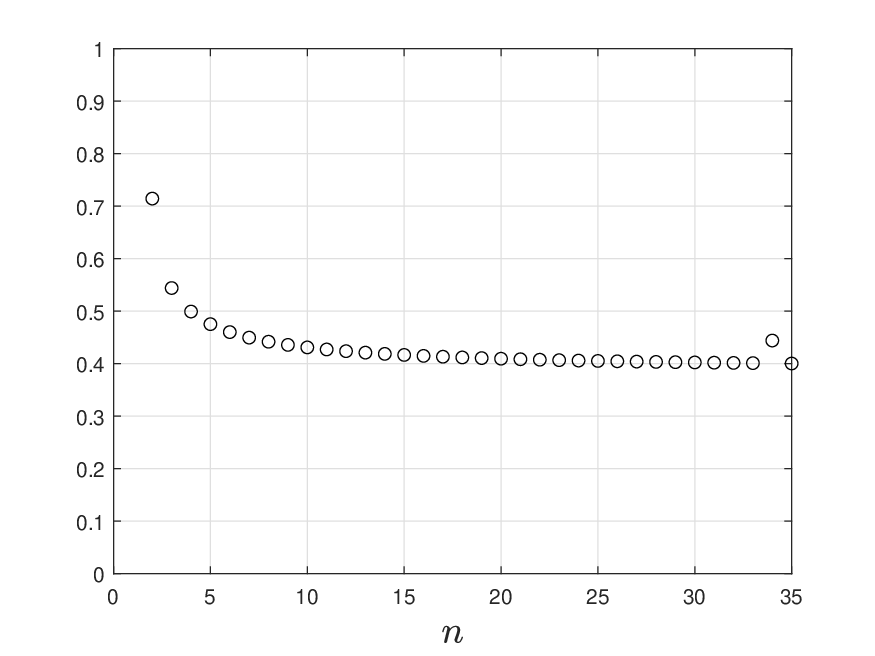}
\caption{The value~$M(n) $ in~\eqref{eq:maxM} as a function of~$n\in\{2,\dots,35\}$.}
\label{fig:mandm}
\end{figure}

%%%%%%%%%%%%%%%%%%%%%%
 \section{Discussion}
%%%%%%%%%%%%%%%%%%%%%%%%%%
We consider a nonlinear eigenvalue optimization problem that arises in an important nonlinear ODE model from systems biology called the ribosome flow model. The main result is  that the optimal solution satisfies a turnpike property. 
To the best of our knowledge, this is the first proof of such a structure in an eigenvalue optimization problem. 
Although we consider a specific problem, our analysis is based on studying the general recursion~\eqref{eq:rec_ai}, and it may be of interest to prove a turnpike structure in other recursions. 

There is a simple intuitive explanation for the turnpike property. To see  this, consider modifying the~$(n+2)\times(n+2)$ matrix~$B$
in~\eqref{eq:B_spec} to the $(n+2)\times(n+2)$
matrix
\be\label{eq:C_spec}
C:=\begin{bmatrix}
    0 & \lambda_0^{-1/2}& 0 &0 &0 &\dots& 0 &0&\lambda_{n+1}^{-1/2}\\
 \lambda_0^{-1/2}& 0 &\lambda_1^{-1/2} &0 &0&\dots& 0 &0&0\\
  0 &\lambda_1^{-1/2} &0 &\lambda_2^{-1/2}&0&\dots& 0 &0&0\\
  &&\vdots\\
   0 &0 &0 &0&0&\dots& \lambda_{n-1}^{-1/2}& 0& \lambda_{n}^{-1/2} \\
  \lambda_{n+1}^{-1/2} &0 &0 &0&0&\dots&0& \lambda_{n}^{-1/2} &0
\end{bmatrix}, 
\ee
that is, we added~$\lambda_{n+1}^{-1/2}$ at entries~$(1,n+2)$ and~$(n+2,1)$. Note that~$C$ has  a kind of cyclic structure. 
Pose the 
problem:~$\min\sigma(C)$ under the constraints~$\lambda_i\geq 0$ and~$\sum_{i=0}^{n+1}\lambda_i\leq n+2$. By symmetry,
the solution is~$\lambda_i=1$ for all~$i$. Thus, we may interpret the turnpike property for the problem with the matrix~$B$ as follows. The zero   entries at~$(1,n+2)$ and~$(n+2,1)$  in~$B$  lead to a transient behavior of the optimal~$\bar \lambda_i$s 
before and after the bulk. In the bulk, they  are close to~$1$ just like in the optimal solution with the matrix~$C$. This interpretation is in line with the turnpike property in some optimal control problems, where the initial and final condition
constraints enforce a transient behavior of the optimal control values before and after the turnpike. The problem of minimizing $\sigma(C)$ then plays the role of the static optimization problem used for determining the turnpike equilibrium in optimal control.

Further  
 topics for    research include the following. First, as noted above, turnpike properties in optimal control is an important and  widely studied topic. As noted in Ref.~\cite{TRELAT201581}: ``As it is well known, the turnpike property is actually due to a general hyperbolicity 
 phenomenon. Roughly speaking, in the neighborhood of a saddle point, any trajectory of a given
hyperbolic dynamical system, which is constrained to remain in this neighborhood in large
time, will spend most of the time near the saddle point. This very simple observation is at the
heart of the turnpike results.''  We have made this statement explicit in the context of our turnpike result in Remark~\ref{remT:Hartman}.
In this respect, it may be of interest to try and formulate the optimization problem studied here as an optimal control problem. One possible approach is to use Remark~\ref{remT:scaling} to  fix~$R>0$ and   re-write~\eqref{eq:eq_psati} and
  \eqref{eq:def_R_steady} as follows. 
 Find a positive control sequence~$u(0),\dots,u(n)$
 that steers the  discrete-time 
  control system
  \[
  x(k+1)=1-\frac{R}{u(k)x(k)}
  \]
 from~$x(0)=1$ to~$x(n+1)=0$
 while minimizing~$\sum_{i=0}^n u(i)$.

\bibliographystyle{IEEEtranS}
\bibliography{refs}

@string{s="Submitted"}

@string{JDE="J. Diff. Eqns."}

@string{PLOSCB="PLOS Computational Biology"}

@string{RSIF="J. Royal Society Interface"}

@string{TAC="IEEE Trans. Autom. Control"}

@string{LAA="Linear Algebra Appl."}

@book{Hartman1982,
    AUTHOR = {Hartman, Philip},
     TITLE = {Ordinary Differential Equations},
   EDITION = {second},
 PUBLISHER = {Birkh\"auser, Boston, MA},
      YEAR = {1982},
     PAGES = {xv+612},
      ISBN = {3-7643-3068-6},
   MRCLASS = {34-01 (58Fxx)},
  MRNUMBER = {658490},
}

@article{solvers_guide,
	author = "Blythe, R. A. and Evans, M. R.",
	journal = "J. Phys. A: Math. Theor.",
	number = "46",
	pages = "R333-R441",
	title = "Nonequilibrium steady states of matrix-product form: a solver's guide",
	volume = "40",
	year = "2007"
}

@ARTICLE{min_spring,
  author={Zarai, Yoram and Margaliot, Michael},
  journal=TAC, 
  title={On Minimizing the Maximal Characteristic Frequency of a Linear Chain}, 
  year={2017},
  volume={62},
  number={9},
  pages={4827-4833},
 }

@article{Zur2020,
author={Hadas Zur and Rachel Cohen-Kupiec and  Sophie Vinokour and Tamir Tuller},
title={Algorithms for ribosome traffic engineering and their potential in improving host cells' titer and growth rate},
journal={Sci. Rep.}, volume={10},pages={21202}, year={2020},
}

@article{Aditi_extended_2022,
year={2022},
author={A. Jain and A. K.  Gupta}, 
title={Modeling transport of extended interacting objects
with drop-off phenomenon}, 
journal={Plos one}, volume={17}, number={5}, pages={e0267858},
}

@article{Ortho_RFM,
author={J. Miller and  M.A. Al-Radhawi and E.D. Sontag},
title={Mediating ribosomal competition by splitting pools},
journal={IEEE Control Systems Letters}, volume={5},
pages={1555-1560},year={2020},
}

@ARTICLE{Aditi_abortions,
author = {A. Jain and A. Gupta},
journal = {IEEE/ACM Trans.  on Computational Biology and Bioinformatics},volume={20}, issue={2}, pages={1600-1605},
title = {Modeling {mRNA} translation with ribosome abortions},
year = {2023},
}

@article{aditi_networks,
author={Aditi Jain and Michael Margaliot   and Arvind Kumar Gupta}, 
year={2022},volume={19},
title={Large-scale {mRNA} translation and the intricate effects of competition for the finite pool of ribosomes},journal={J. R. Soc. Interface}, pages={2022.0033},
}

@article{randon_rfm,
author={M. Margaliot and W. Huleihel 
and T. Tuller},
title={Variability in {mRNA} translation: a random matrix theory approach},
journal={Sci. Rep.},
volume={11},  year={2021},
}

@article{down_reg_mrna,
title={Optimal Down Regulation of {mRNA} Translation},
author={Yoram Zarai and Michael Margaliot   and  Tamir Tuller},
journal={Sci. Rep.}, 
volume={7}, number={41243}, year={2017},
}

@ARTICLE{rfmr_2015,
  author={Raveh, Alon and Zarai, Yoram and Margaliot, Michael and Tuller, Tamir},
  journal={IEEE/ACM Transactions on Computational Biology and Bioinformatics}, 
  title={Ribosome Flow Model on a Ring}, 
  year={2015},
  volume={12},
  number={6},
  pages={1429-1439},
}

@article{alexander2017,
author={Yoram Zarai 
and   Alexander Ovseevich  
and  Michael Margaliot},
year={2017},
title={Optimal Translation Along a Circular {mRNA}},
journal={Sci. Rep.},
pages={9464}, volume={7},number={1},
}

@article{rfm_sense,
title={Sensitivity of {mRNA} Translation},
author={Gilad Poker and Michael Margaliot and Tamir Tuller},
journal={Sci. Rep.}, 
volume={5}, number={12795}, year={2015},
}

@article{rfm_max,
author={G. Poker and Y. Zarai and M.    Margaliot  and T. Tuller},
title={Maximizing  protein
translation rate in the nonhomogeneous ribosome flow model:
 A convex optimization approach},
 journal=RSIF, year={2014},
pages={20140713},
volume = {11},
number = {100},
}

@article{EYAL_RFMD1,
title = "Ribosome flow model with different site sizes",
 journal="{SIAM} J. Applied Dynamical Systems",
year = "2020",  
  volume={19},
number={1},
pages={541-576},
author = "Eyal Bar-Shalom and  Alexander Ovseevich and  Michael Margaliot",
}

@Article{Margaliot2014Entrain,
    author = {Margaliot, Michael AND Sontag, Eduardo D. AND Tuller, Tamir},
    journal = {PLOS ONE},
    publisher = {Public Library of Science},
    title = {Entrainment to Periodic Initiation and Transition Rates in a Computational Model for Gene Translation},
    year = {2014},
    volume = {9},
    pages = {e96039},
    number = {5},
}

@book{HornJohnson2013MatrixAnalysis,
    title="Matrix Analysis",
    author="R. A. Horn and C. R. Johnson",
publisher="Cambridge University Press",
    year="2013",
    edition="2",
}

@article{margaliot2012stability,
	title={Stability analysis of the ribosome flow model},
	author={Margaliot, Michael and Tuller, Tamir},
	journal={IEEE/ACM Trans.   Computational Biology and   Bioinformatics},
	volume={9},
	number={5},
	pages={1545-1552},
	year={2012},
}

@book{book_toeplitz,
author = {Böttcher, Albrecht and Grudsky, Sergei M.},
title = {Spectral Properties of Banded Toeplitz Matrices},
publisher = {Society for Industrial and Applied Mathematics},
year = {2005},
}

@ARTICLE{Raveh2016,
  author =       {Raveh, A. and Margaliot, M. and  Sontag, E.D. and Tuller, T.},
  title =        {A model for competition for ribosomes in the cell},
  journal =      RSIF,
  year =         {2016},
  volume =       {13},
  number =       {116},
  pages =        {20151062},
}

@article{ribo_jam2021,
    author ="S. Subramaniam" ,
    title = "Ribosome traffic jam in neurodegeneration: decoding hurdles in {Huntington} disease",
    journal = "Cell Stress",
    year = "2021",volume="5", number="6",pages="86-88",
}

@article{fierce_compete,
	author = {Katz, Rami and Attias, Elad and Tuller, Tamir and Margaliot, Michael},
	title = {Translation in the cell under fierce competition for shared resources: a mathematical model},
	year = {2022},
volume={19},journal=RSIF,
pages={20220535}, 
}

@article{RFM_NEGATIVE_FEEDBACK,
	author = {Aliza Ehrman and Thomas Kriecherbauer and Lars Gr\"une  and Michael Margaliot},
	title = {Negative feedback and oscillations in a model
for {mRNA} translation},volume={22},
pages={20250338},
	year = {2025},
 journal=RSIF,
}

@article{nani,
author={Nanikashvili, Itzik
and Zarai, Yoram
and Ovseevich, Alexander
and Tuller, Tamir
and Margaliot, Michael}, year={2019},
title={Networks of ribosome flow models for modeling and analyzing intracellular traffic},
journal={Sci. Rep.},
page={1703},volume={9},number={1},
}

@article{allgower_RFM,
  url = {https://arxiv.org/abs/1610.03986}, 
  author = {Halter, Wolfgang and Montenbruck, Jan Maximilian and Allgower, Frank},
   title = {Geometric stability considerations of the ribosome flow model with pool},
  year = {2016},
}

@article{dyson_chain,
  title = {The Dynamics of a Disordered Linear Chain},
  author = {Dyson, Freeman J.},
  journal = {Phys. Rev.},
  volume = {92},
  issue = {6},
  pages = {1331-1338},
   year = {1953},
}

@article{BOYD199363,
title = {Method of centers for minimizing generalized eigenvalues},
journal = LAA,
author={Stephen Boyd and Laurent {El Ghaoui}}, 
volume = {188-189},
pages = {63-111},
year = {1993}
}

@article{Overton1988,
    author = "M. L. Overton",
    title = "On minimizing the maximum eigenvalue of a symmetric matrix",
    journal ="J. Matrix Anal. Appl.",
    year = "1988",volume="9",pages="256-268"
}

@article{TRELAT201581,
title = {The turnpike property in finite-dimensional nonlinear optimal control},
journal = JDE,
volume = {258},
number = {1},
pages = {81-114},
year = {2015},
author = {Emmanuel Trélat and Enrique Zuazua}
}

@article{reuveni2011genome,
  title={Genome-scale analysis of translation elongation with a ribosome flow model},
  author={Reuveni, Shlomi and Meilijson, Isaac and Kupiec, Martin and Ruppin, Eytan and Tuller, Tamir},
  journal=PLOSCB,
  volume={7},
  number={9},
  pages={e1002127},
  year={2011},
}

@InCollection{FauG22,
  author    = {T. Faulwasser and L. Gr\"une},
  booktitle = {Numerical Control: Part A},
  publisher = {Elsevier},
  title     = {Turnpike properties in optimal control},
  year      = {2022},
  editor    = {E. Tr\'elat and E. Zuazua},
  pages     = {367--400},
  volume    = {24},
  doi       = {10.1016/bs.hna.2021.12.011},
}

@Article{Grue22,
  author    = {Gr\"{u}ne, Lars},
  journal   = {IEEE Control Systems Magazine},
  title     = {Dissipativity and Optimal Control: Examining the Turnpike Phenomenon},
  year      = {2022},
  issn      = {1941-000X},
  month     = apr,
  number    = {2},
  pages     = {74--87},
  volume    = {42},
  doi       = {10.1109/mcs.2021.3139724},
  publisher = {Institute of Electrical and Electronics Engineers (IEEE)},
}

@Article{PorZ13,
  Title                    = {Long time versus steady state optimal control},
  Author                   = {A. Porretta and E. Zuazua},
  Journal                  = {SIAM J. Control Optim.},
  Year                     = {2013},
  Number                   = {6},
  Pages                    = {4242--4273},
  Volume                   = {51}
}

@Book{DoSS58,
  author     = {Dorfman, R. and Samuelson, P. A. and Solow, R. M.},
  publisher  = {McGraw-Hill},
  title      = {Linear programming and economic analysis},
  year       = {1958},
  address    = {New York-Toronto-London},
  series     = {A Rand Corporation Research Study},
  mrclass    = {90.10},
  mrnumber   = {0128543},
  mrreviewer = {N. N. Vorob\cprime ev},
  pages      = {ix+525},
}

@Article{Rams28,
  author    = {Ramsey, F. P.},
  journal   = {The Economic Journal},
  title     = {A mathematical theory of saving},
  year      = {1928},
  number    = {152},
  pages     = {543--559},
  volume    = {38},
  publisher = {JSTOR},
}
%%%%%
%%%%%
\end{document}